\documentclass[12pt]{amsart}

\usepackage[top=100pt,bottom=60pt,left=96pt,right=92pt]{geometry}

\usepackage{hyperref}
\usepackage{amsmath,amsthm,amssymb,amsfonts,enumerate,xcolor,layout,tikz,subcaption,enumitem,fullpage, verbatim}
\usepackage{faktor}
\usepackage{tikz}
\usetikzlibrary{matrix}
\usepackage{tikz-cd}
\usepackage{extarrows}

\newtheorem{theorem}{Theorem}[section]
\newtheorem{proposition}[theorem]{Proposition}
\newtheorem{corollary}[theorem]{Corollary}
\newtheorem{lemma}[theorem]{Lemma}

\theoremstyle{definition}
\newtheorem{definition}[theorem]{Definition}

\theoremstyle{plain}

\definecolor{linkblue}{rgb}{0,0,.6}
\definecolor{citered}{rgb}{.7,0,0}
\hypersetup{colorlinks =true, linkcolor=linkblue, citecolor = citered, urlcolor=linkblue}

\def\C{{\mathbb C}}
\def\N{{\mathbb N}}

\def\R{{\mathbb R}}

\def\Z{{\mathbb Z}}
\def\ft{{\mathfrak{t}}}

\def\V{{\mathcal V}}
\def\E{{\mathcal E}}
\def\Tilde{\widetilde}
\def\del{\partial}

\newcommand{\De}{\Delta}
\newcommand{\om}{\omega}

\newcommand{\GL}{\textrm{GL}}

\newcommand{\relint}{\operatorname{relint}}

\newcommand{\fh}{\mathfrak{h}}
\newcommand{\Y}{Y =  T \times_H \fh^\circ \times \C^{h+1}}

\begin{document}

\title{Tall Complexity One Spaces with k-colorable Skeleton}

\author{Yichen Liu}
\date{\today}
\address{Mathematics Department\\
	University of Illinois at Urbana-Champaign\\Champaign, Illinois 61801}
\email{yichen23@illinois.edu}

\begin{abstract}
Tall complexity one $T$-spaces are Hamiltonian $T$-spaces $(M,\om,\Phi)$ such that $\frac{1}{2}\dim M -\dim T=1$ and the symplectic quotient at each moment value is a surface. The skeleton of a complexity one $T$-space is an important invariant in the classification and encodes the information about non-generic orbits. In this paper, we study properties of the skeleton of a compact, connected tall complexity one $T$-spaces. We prove that when the skeleton is $k$-colorable, i.e., when it can be partitioned into $k$ closed and open subsets such that the orbital moment map is injective on each of them, its information can be recovered by the one-skeleton (the set of non-generic orbits whose dimension is at most one). We also prove that for any cloesd and open subset of the skeleton on which the orbital moment map is injective, one can construct a symplectic toric $(T\times S^1)$-manifold whose underlying complexity one $T$-space has the skeleton isomorphic to this subset.
\end{abstract}

\maketitle

\section{Introduction}
Let $T$ be a torus with Lie algebra $\mathfrak{t}$ and dual space $\ft^*$. Recall that we say $T$ acts on a symplectic manifold $(M,\omega)$ \textit{in a Hamiltonian fashion} if there exists an $T$-invariant map $\Phi: M \to \ft^*$ such that $d \langle \Phi, X \rangle = \omega(X^\#, \cdot)$ for all $X \in \ft$, where $X^\#$ is the fundamental vector field on $M$ corresponding to $X$. We call the tuple $(M,\om,\Phi)$ a {\bf Hamiltonian $T$-space}. If not otherwise stated, we always assume the action is effective\footnote{An action is effective if every non-identity element acts non-trivially.} in this paper. The {\bf complexity} of $(M,\om,\Phi)$ is $\frac{1}{2} \dim M - \dim T$; it is half the dimension of the reduced space $\Phi^{-1}(\alpha)/T$ at a regular value $\alpha \in \Phi(M)$. Hamiltonian $T$-spaces of complexity zero are known as symplectic toric manifolds. Such spaces are classified by their moment images (see Theorem~\ref{thm:Delzant}). 

Let $(M,\omega,\Phi)$ be a $2n$–dimensional Hamiltonian $T$–manifold. Recall that the Liouville measure on $M$ is given by integrating the volume form $\frac{\omega^n}{n!}$ with respect to the symplectic orientation and that the {\bf Duistermaat–Heckman measure} is the pushforward of the Liouville measure by the moment map. Recall that the symplectic slice at $p \in M$ is the symplectic vector space $(T_p \mathcal{O})^\omega / (T_p \mathcal{O} \cap (T_p \mathcal{O})^\omega)$, where $\mathcal{O}$ is the $T$-orbit of $p$. Let $H \subset T$ be the stabilizer group of $p$. The isotropy representation of $H$ on $T_p M$ induces a representation on the quotient space, called the {\bf slice representation.} Since the symplectic slice inherits the complex structure from $T_pM$, the slice representation decomposes into one-dimensional complex irreducible representations each of which is determined by a weight vector $\eta_i \in \mathfrak{h}^*$. Under the identification $(T_p \mathcal{O})^\omega / (T_p \mathcal{O} \cap (T_p \mathcal{O})^\omega) \cong \C^k$, we can express the slice representation as a group homomorphism $\rho: H \to (S^1)^k$ such that $h.(z_1,\ldots,z_k) = (\rho_1(h)z_1,\ldots,\rho_k(h)z_k)$, where the differential of $\rho_i$ is $\eta_i$.

Fix a complexity one $T$-space $(M,\om,\Phi)$, we say a point $p \in M$ is {\bf tall} if every open neighborhood of the orbit through $p$ in the reduced space $\Phi^{-1}(\Phi(p))/T$ contains more than one orbit. A complexity one $T$-space $(M,\om,\Phi)$ is {\bf tall} if every point in $M$ is tall. 
An orbit $\mathcal{O}_p \subset M$ is {\bf exceptional} if every
nearby orbit in the moment fiber $\Phi^{-1}(\Phi(p))$ has strictly smaller stabilizer. Correspondingly, a point in an exceptional orbit is called an exceptional point.
The {\bf skeleton} of $(M,\omega,\Phi)$ is the set of tall exceptional orbits, considered as a subspace of $M/T.$ The {\bf one-skeleton} is the set of tall exceptional orbits that are at most one dimensional. 
Given $k \in \N$, we say that the (one-)skeleton of $M$ is {\bf $\mathbf k$-colorable} if the (one-)skeleton is the disjoint union of $k$ (possibly empty) clopen\footnote{More explicitly, each subset is closed and open in the subset topology on the (one-)skeleton.} subsets
so that the orbital moment map\footnote{With the slight abuse of notation, we denote both the moment map and the orbital moment map (the map induced by the moment map on the orbit space) by $\Phi$.} $\Phi \colon M/T \to \ft^*$
restricts to an injection on each subset.
Two complexity one $T$-spaces have {\bf isomorphic (one-)skeletons} if
there exists a homeomorphism from the (one-)skeleton of one space to that of another which intertwines the moment maps and takes each exceptional orbit to an exceptional orbit with the same stabilizer and isomorphic slice representations. The $k$-colorability is preserved by isomorphisms between (one-)skeletons. Finally, the $T$-action on $M$ \textbf{extends to a toric action} if there exists a torus $\Tilde T$ with $\dim(\Tilde T) = \frac{1}{2}\dim (M)$ that acts effectively on $M$ in a Hamiltonian fashion and an inclusion $T \hookrightarrow \Tilde T$ so that the $\Tilde T$-action restricts to the given $T$-action.

Complexity one spaces were first studied in dimension four by Karshon~\cite{Kar99}. Later, in a series of papers~\cite{KT01,KT03,KT14}, Karshon and Tolman classified tall complexity one spaces in any dimensions. These spaces also demonstrate new phenomena: Tolman~\cite{Tolman98} constructed a complexity one space that admits no invariant K\"ahler structure, while all symplectic toric manifolds are K\"ahler~\cite{De88}. 

In~\cite{LPT}, the author, Palmer and Tolman prove that a complexity one $T$-space has the same combinatorial invariants (the skeleton, the Duistermaat Heckman measure) as a complexity one $T$-space obtained by restricting the toric action on a symplectic toric manifold if and only if the skeleton is $2$-colorable. In this paper, we will focus on properties of the skeleton of tall complexity one $T$-spaces. Our theorems should still hold true if we assume the moment map to be proper as a map to a convex open subset of $\ft^*$ instead of assuming the manifold to be compact. For simplicity, we will just state and prove our main results in the compact case, but prove some auxiliary lemmas in the proper case. We first notice that the definition of $k$-colorability can be reduced to the $k$-colorability on the one-skeleton.

\begin{theorem}\label{thm:k-colorability}
Let $(M,\omega, \Phi)$ be a compact connected tall complexity one $T$-space. The skeleton of $(M,\omega,\Phi)$ is $k$-colorable if and only if the one-skeleton of $(M,\omega,\Phi)$ is $k$-colorable.
\end{theorem}
Moreover, we apply similar techniques as in~\cite{LPT} to prove that as long as the orbital moment map is injective on a clopen subset $E$ of the skeleton of a tall complexity one $T$-space, we can construct a symplectic toric $(T\times S^1)$-manifold whose underlying complexity one $T$-space has the skeleton isomorphic to $E$. Charton-Kessler-Tolman~\cite[Question 1.5]{CKT} asked whether trivial paintings (an invariant of a tall complexity one $T$-space) imply the existence of an invariant K\"ahler structure. Our theorem below will be useful in answering their question by providing an invariant K\"ahler structure compatible with $E$. 

\begin{theorem} \label{thm:toric}
Let $(M,\omega,\Phi)$ be a compact connected tall complexity one $T$-space. Let $E$ be a nonempty clopen (closed and open) subset of the skeleton so that the orbital moment map $\Phi \colon E \to \ft^*$ is injective on $E$. Then there exists a connected tall complexity one $T$-space $(M',\om',\Phi')$ with proper moment map such that the following holds:
\begin{enumerate}
\item $\Phi(M) = \Phi'(M')$,
\item the skeleton of $M'$ is isomorphic to $E$, and
\item the $T$-action on $M'$ extends to a toric action.
\end{enumerate}
\end{theorem}

As a consequence of Theorem~\ref{thm:toric}, we prove that for $k$-colorable tall complexity one $T$-spaces, the one-skeleton decides the skeleton.

\begin{theorem}\label{thm:1-skeleton}
Let $(M_1,\omega_1,\Phi_1)$ and $(M_2,\omega_2,\Phi_2)$ be two compact connected tall complexity one $T$-spaces that both have $k$-colorable skeletons. If $M_1$ and $M_2$ have isomorphic one-skeletons, then they have isomorphic skeletons.
\end{theorem}

The structure of the paper is as follows. In Section~\ref{sec:background}, we review some properties of complexity one spaces. In Section~\ref{sec:k-colorability}, we prove Theorem~\ref{thm:k-colorability}. In Section~\ref{sec:cx-1-toric}, we focus on a special family of complexity one spaces: the ones obtained by restricting the toric action on a symplectic toric $(T \times S^1)$-manifold to the action of the subtorus $T \times \{1\}$. This section prepares us for the proof of Theorem~\ref{thm:toric} in Section~\ref{sec:extension}. Finally, in Section~\ref{sec:graph}, we build a graph to encode the information of the one-skeleton of a tall complexity one $T$-space with $k$-colorable skeleton and prove Theorem~\ref{thm:1-skeleton}.

{\bf Acknowledgements.} The author would like to thank Susan Tolman for her patient guidance and invaluable suggestions. The author also wants to thank Joseph Palmer, Yael Karshon for many helpful discussions.

\section{Background}\label{sec:background}
In this section, we recall the local normal form theorem and use it to prove a few facts about exceptional orbits in complexity one spaces.

Let $T = (S^1)^k$ be a torus, $\mathfrak{t} \cong \R^k$ be its Lie algebra and $\mathfrak{t}^* \cong \R^k$ be the dual of the Lie algebra. Let $H \subset T$ be a closed subgroup and $\mathfrak{h} \subset \mathfrak{t}$ its Lie algebra and $\mathfrak{h}^\circ \subset \mathfrak{t}^*$ the annihilator of $\mathfrak{h}$ in $\mathfrak{t}^*$. Let $H$ act on $\C^n$ by the composition of a group homomorphism $\rho= (\rho_1,...,\rho_n) \colon H \to (S^1)^n$ followed by component-wise multiplication. In coordinates, 
\[h.(z_1,...,z_n) = (\rho_1(h)z_1,...,\rho_n(h)z_n).\]
This action has a moment map given by 
$\Phi_H (z_1,...,z_n) = \frac{1}{2} \sum_{j=1}^n |z_j|^2 \eta_j$,
where the $\eta_j$, called the \emph{weights} of the action, are the differentials of $\rho_j$, regarded as elements in $\mathfrak{h}^*$. Once and for all, we fix an inner product on $\mathfrak{t}$ and identify $\mathfrak{h}^*$ as a subspace of $\mathfrak{t}^*$.

Consider the space
$Y = T \times_H \fh^\circ \times \C^n$
with $T$ acting on the left. The action of $T$ is Hamiltonian,
and there is a moment map
$\Phi_Y([t,\nu,z]) = \Phi_H(z) + \nu.$
We will call this the \emph{local model} because of the following local normal form by Guillemin-Sternberg \cite{GS84}.

\begin{theorem}[Local normal form~\cite{GS84}]\label{thm:lnf}
Let $(M,\omega,\Phi)$ be a Hamiltonian $T$-space where $T$ acts effectively. Given a point $p \in M$ with stabilizer group $H \subset T$ and slice representation $\rho$, there exists a neighborhood of the orbit $T \cdot p$ that is equivariantly symplectomorphic to a neighborhood of the orbit $\{[t,0,0]\}$ in the model $Y:= T \times_H \fh^\circ \times \C^n$.
\end{theorem}

The next lemma establishes that the moment image of the local model agrees locally with the moment polytope. It is a special case of \cite[Theorem 6.5]{Sjamaar}; see also \cite[Theorems 1.2, 4.3, 6.1 and 6.2]{LMTW}. 

\begin{lemma}\label{lem:local-cone}
Let  $(M,\omega,\Phi)$ be a connected Hamiltonian $T$-space so that $\Phi$ is proper as a map to a convex open subset of $\ft^*$.
Given $p \in M$,  let $Y_p$ be the local model associated to $p$ with moment map $\Phi_p$, normalized so that $\Phi_p([1,0,0]) = \Phi(p)$.
Then there exists an open neighborhood $W$ of $\Phi(p)$ such that
$\Phi(M) \cap W = \Phi_p(Y_p) \cap W$.
\end{lemma}


We call a model $Y$ a \textit{complexity one local model} if $(Y,\omega,\Phi_Y)$ is a complexity one space. In this case, $\Y$, where $h =\dim H$. The next lemma follows immediately from \cite[Lemmas 5.4 and 5.8]{KT01} and it plays an essential role in this section.

\begin{lemma}\label{lem:defining-monomial}
    Let $\Y$ be a tall complexity one local model with moment map $\Phi_Y([t,\nu,z]) = \frac{1}{2} \sum_{j=0}^{h} |z_j|^2 \eta_j + \nu$. There exists a unique vector $\xi  \in \Z_{\geq 0}^{h+1}$ such that the following sequence is exact:
\begin{equation}\label{ses} 1 \to H \xrightarrow{\rho} (S^1)^{h+1} \xrightarrow{P} S^1 \to 1,
\end{equation}
where $\rho:H \to (S^1)^{h+1}$ is the slice representation at $[1,0,0]$ and $P(t_0,...,t_{h}) = \prod_{j=0}^{h} t_j^{\xi_j}$. Moreover, $\sum a_j \eta_j=0$ if and only if $a = \lambda \xi$ for some $\lambda \in \R$.
\end{lemma}
The monomial $P$ is called the {\bf defining monomial of the complexity one local model $Y$.}  We now give a criterion for exceptional points in terms of the exponents in the defining monomials.

\begin{lemma}\label{lem:exceptional-orbits}
    Let $Y = T \times_H \fh^\circ \times \C^{h+1}$ be a tall complexity one local model with moment map $\Phi_Y([t,\nu,z]) = \frac{1}{2} \sum_{k=0}^{h} |z_k|^2 \eta_k + \nu $ and let $\xi_0,\ldots,\xi_h \in \Z_{\geq 0}$ be the exponents of the defining monomial of $Y$. A point $[t, \nu, z] \in Y$ is exceptional if and only if
$\sum_{z_k = 0} \xi_k > 1$, where the sum is over all $k \in \{0,\dots,h\}$ such that $z_k = 0$.
\end{lemma}

\begin{proof}
    We compare the stabilizer groups of points in the same fiber. If we identify $H$ with its image under the embedding $\rho$ to $(S^1)^{h+1}$, the stabilizer group of $[t,\nu,z]$ is the intersection of $H$ with $\{(t_0,...,t_{h}): t_i =1 \textrm{ if } z_i \neq 0\}$. By the short exact sequence~\eqref{ses}, $\rho(H) = \ker P$, so 
    \[\textrm{stab}_{[t,\nu,z]} \cong \{(t_0,...,t_{h}): \prod_{z_i=0} t_i^{\xi_i} = 1 \textrm{ and } t_i =1 \textrm{ if } z_i \neq 0 \}.\]

    Since $[t,\nu,z]$ is a tall point, there exists a point $[t',\nu,w]$ such that $z \neq w$ and $\sum_{k=0}^{h} |z_k|^2 \eta_k =\sum_{k=0}^{h} |w_k|^2 \eta_k$. By Lemma~\ref{lem:defining-monomial}, there exists an $s \in \R \setminus\{0\}$ such that $|w_k|^2 -|z_k|^2 = s \xi_k$ for all $k=0,...,h$. Moreover, we can choose $[t',\nu,w]$ such that $|w_k| = |z_k|$ whenever $\xi_k=0$ and $w_k \neq 0$ whenever $\xi_k \neq 0$. Hence, $w_k=0$ if and only if $z_k=\xi_k=0$, so $\prod_{w_i = 0} t_i^{\xi_i}  =1$. Hence,
    \[\textrm{stab}_{[t',\nu,w]}\cong\{(t_0,...,t_{h}): t_i =1 \textrm{ if } z_i \neq 0 \textrm{ or } \xi_i \neq 0 \}.\]

    $[t,\nu,z]$ is exceptional if and only if $\textrm{stab}_{[t'\nu,w]} \subsetneq \textrm{stab}_{[t,\nu,z]}$. It is straightforward to check that $\textrm{stab}_{[t',w,\nu]} \subseteq \textrm{stab}_{[t,z,\nu]}$. The equality holds if and only if $\prod_{z_i=0} t_i^{\xi_i} = 1 $ implies $t_i =1$ for all $i$ such that $\xi_i \neq 0$. Notice that  $\prod_{z_i=0} t_i^{\xi_i} =  \prod_{z_i=0,\xi_i \neq0} t_i^{\xi_i}$, so the two stabilizer groups are equal if and only if one of the following holds:
    \begin{itemize}
        \item $z_i=0$ for all $i$ such that $\xi_i =0$, or
        \item there exists a unique $i \in \{0,\ldots,h\}$ such that $|\xi_i|=1$ and $z_i=0$.
    \end{itemize}
    
    Hence, $[t,\nu,z]$ is exceptional if and only if there exist $i,j$ such that $z_i=z_j=0$ and $\xi_i,\xi_j \neq 0$ or there exists $i$ such that $z_i=0$ and $\xi_i>1$. Therefore, $[t,\nu,z]$ is exceptional if and only if $\sum_{z_k=0} \xi_k >1$.
\end{proof}

Lemma~\ref{lem:exceptional-orbits} immediately implies the following.
\begin{corollary}\label{cor:model}
     Let $Y = T\times_H \fh^\circ \times \C^{h+1}$ be a tall complexity one local model. 
    \begin{enumerate}
        \item If $[1,0,0]$ is non-exceptional, then $Y$ contains no exceptional points.
        \item If $[1,0,0]$ is exceptional, then the set of points in $Y$ fixed by $H$ is $Y^H = T \times_H \fh^\circ \times \{0\}$.
    \end{enumerate}
\end{corollary}

Next, we show that the skeleton of each tall complexity one local model is $1$-colorable.

\begin{lemma}\label{lem:2-colorable-local-model}
Let $Y = T \times_H  \mathfrak{h}^\circ \times \C^{h+1}$ be a tall complexity one local model with moment map $\Phi_Y([t,\nu,z]) = \frac{1}{2} \sum_{j=0}^{h} |z_j|^2 \eta_j + \nu$. Let $\Sigma_Y$ be its skeleton. Then $\Sigma_Y$ is connected and the orbital moment map $\Phi_Y: Y/T \to \ft^*$ restricts to a closed injection on $\Sigma_Y$. 
\end{lemma}

\begin{proof}
   Without loss of generality, we assume that $[1,0,0]$ is exceptional, otherwise by Corollary~\ref{cor:model}, $\Sigma_Y = \emptyset$. Let $\xi_0,\ldots, \xi_h \in \Z_{\geq 0}$ be the exponents of the defining monomial of $Y$. By Lemma~\ref{lem:defining-monomial}, $\sum \xi_j \eta_j =0$. By Lemma~\ref{lem:exceptional-orbits}, 
     $\Sigma_Y = \{[t,\nu,z] \in Y/T : \sum_{z_i = 0 } \xi_i > 1\}$.
     Furthermore, $\Sigma_Y$ is path-connected, as $\gamma(s)=[t,sz,s\nu]$ connects $[t,z,\nu] \in \Sigma_Y$ to $[1,0,0] \in \Sigma_Y$.

    Given two orbits $ [t,\nu,z],[t',\nu',w] \in \Sigma_Y$ such that $\Phi_Y([t,\nu,z]) =\Phi([t',\nu',w])$, then $\nu = \nu'$ and $0 =  \sum_{j} (|z_j|^2-|w_j|^2) \eta_j$. By Lemma~\ref{lem:defining-monomial}, there exists $\lambda \in \R$ such that $|z_j|^2-|w_j|^2 = \lambda\xi_j$ for all $j = 0,\ldots,h$. By assumption, there exist $i,k$ such that $z_i=w_k=0$ and $\xi_i,\xi_k >0$. Hence, $-|w_i|^2 = \lambda \xi_i$ and $|z_k|^2 =\lambda \xi_k$. Since $\xi_i,\xi_k >0$, these two equations imply that $\lambda =0$. Hence, $|w_j|=|z_j|$ for all $j$, so we conclude that $\Phi_Y$ restricts to an injection on $\Sigma_Y$. By Lemma 6.2 in~\cite{KT01}, the orbital map $F:= ({\Phi}_Y,{P}): Y/T \to \ft^* \times \C$ is a homeomorphism onto its image, where $P([t,\nu,z]) = \prod_{j=0}^{h} z_j^{\xi_j}$. By Lemma~\ref{lem:exceptional-orbits}, $P(\Sigma_Y)=0$. Hence, $\Phi_Y$ restricts to a homeomorphism from $\Sigma_Y$ to its image. Moreover,  Lemma~\ref{lem:exceptional-orbits} also implies that $\Phi_Y(\Sigma_Y)$ is a finite union of closed cones, so $\Phi_Y(\Sigma_Y)$ is closed in $\ft^*$ and it follows that $\Phi_Y:Y/T \to \ft^*$ restricts to a closed injection on $\Sigma_Y$.
\end{proof}

\begin{lemma}\label{lem:exceptional-saturated-new} 
Let $(M,\om,\Phi)$ be a connected tall complexity one $T$-space with proper moment map. Let $E \subset M/T$ be a clopen subset of the skeleton. Fix an $\alpha \in \Phi(M)$ and an orbit $\mathcal{O} \in \Phi^{-1}(\alpha)$ such that $\mathcal{O} \in E$ if $\alpha \in \Phi(E)$ and $\mathcal{O}$ is non-exceptional otherwise. Let $Y$ be the local model associated to $\mathcal{O}$ with moment map $\Phi_Y:Y \to \ft^*$. Let $\Sigma_Y$ be the skeleton of the model $Y$. Then there exist neighborhoods $U \subset \Phi(M)$ of $\alpha$ and $V \subset M/T$ of $\mathcal{O}$ such that $E \cap \Phi^{-1}(U) \cap V$ is isomorphic to $\Sigma_Y \cap \Phi_Y^{-1}(U)$.
\end{lemma}

\begin{proof}
    By Theorem~\ref{thm:lnf}, there exists a neighborhood $V \subset M/T$ of $\mathcal{O}$ that is equivariantly symplectomorphic to a neighborhood $V' \subset Y/T$ of the orbit $\{[t,0,0]\}$. Hence, the set of exceptional orbits in $V/T$ is isomorphic to $\Sigma_Y \cap V'/T$.

    If $\mathcal{O}$ is non-exceptional, then Corollary~\ref{cor:model} implies that $\Sigma_Y = \emptyset$. Moreover, Lemma~\ref{lem:exceptional-orbits} implies that the set of exceptional points in $M$ is a closed subset of $M$. By shrinking $V$ and $V'$ if necessary, we may assume that $V$ contains no exceptional orbits because $\mathcal{O}$ is not exceptional, so it follows that $E \cap V= \emptyset$. Hence, any open neighborhood $U$ of $\alpha$ will make both $E \cap \Phi^{-1}(U) \cap V$ and $\Sigma_Y \cap \Phi_Y^{-1}(U)$ empty sets.
    
    If $\mathcal{O} \in E$, since $E$ is an open subset of the skeleton, we may assume that the set of exceptional orbits in $V$ is equal to $E \cap V$ and hence $E \cap V$ is isomorphic to $\Sigma_Y \cap V'$. 
    By Lemma~\ref{lem:local-cone}, there exists a neighborhood $W$ of $\alpha$, open in both $\Phi(M)$ and $\Phi_Y(Y)$. Since both $\Phi, \Phi_Y$ are open maps onto their images (c.f.~\cite[Theorem 5.4 and Example 5.5]{Sjamaar} or ~\cite[Section 7]{BjoKar2010}), after possibly shrinking $W$ we can assume that $\Phi^{-1}(W) \cap V= V$ and $\Phi_Y^{-1}(W) \cap V' =  V'$, so $E \cap \Phi^{-1}(W) \cap V = E \cap V$ is isomorphic to $\Sigma_Y \cap V' = \Sigma_Y \cap \Phi_Y^{-1}(W) \cap V'$. Finally, by Lemma~\ref{lem:2-colorable-local-model}, $\Phi_Y$ is injective on $\Sigma_Y$, so $\Phi_Y^{-1}(\alpha) \cap \Sigma_Y \subset V'$. Moreover, the same lemma implies that $\Phi_Y$ restricts to a closed map from $\Sigma_Y$ to $\ft^*$, so there exists an open neighborhood $W' \subset \ft^*$ such that $\Phi_Y^{-1}(W') \cap \Sigma_Y \subset V'$. Let $U = W \cap W'$, then $E \cap \Phi^{-1}(U) \cap V$ is isomorphic to $\Sigma_Y \cap \Phi_Y^{-1}(U)$.
\end{proof}

\begin{lemma}\label{lem:toric-exceptional}
    Let $(M,\omega,\Phi)$ be a complexity one $T$-space with proper moment map. Let $H \subseteq T$ be a codimension $k$ subgroup and $N \subset M^H$ be a $2k$-dimensional connected component such that $T/H$ acts effectively on $N$. Then every point in $N$ is exceptional.
\end{lemma}

\begin{proof}
    Points in an orbit $T \cdot p \subset M$ with stabilizer $K \subseteq T$ are exceptional exactly if
$T \cdot p$ is a connected component of $M^K \cap \Phi^{-1}(\Phi(p))$.
Since $\dim(T/H)=k$ and $\dim(N) = 2k$, this implies that $\Phi^{-1}(\beta) \cap N$ contains exactly one orbit for all $\beta \in \Phi(N)$.
Therefore, every point in $N$ is exceptional.
\end{proof}

The following corollary is an immediate consequence of item (2) of Corollary~\ref{cor:model}, Lemma~\ref{lem:toric-exceptional} and~\cite[Corollary 2.6]{KT14}.

\begin{corollary}\label{cor:exceptional-component}
    Let $(M,\omega,\Phi)$ be a tall complexity one $T$-space with proper moment map. Let $p \in M$ be an exceptional point with stabilizer group $H$. Let $N$ be the connected component of $M^H$ that contains $p$. Then every point in $N$ is exceptional and $N$ is a connected symplectic toric $(T/H)$-manifold.
\end{corollary}

\section{Proof of Theorem~\ref{thm:k-colorability}}\label{sec:k-colorability}
In this section, we prove Theorem~\ref{thm:k-colorability}. We first prove that the orbital moment map is injective on a connected component of the skeleton if and only if it is injective on the intersection of this component with the one-skeleton. We need the following two lemmas.

\begin{lemma}\label{lem:affine-linelocalmodel} 
Let $Y = T \times_{H} \mathfrak{h}^\circ \times \C^{h+1}$ be a tall complexity one local model with moment map $\Phi_Y([t,\nu,z]) = \frac{1}{2}\sum_{i=0}^h |z_i|^2 \eta_i + \nu$ such that $[1,0,0]$ is exceptional. Let $\mathcal L \subset \mathfrak t^*$ be an affine line such that $\mathcal L \subset \Phi_Y(Y)$ and assume that $\mathcal L' \cap \mathfrak h^\circ = \{0\}$, where $\mathcal L'$ is the line through the origin which is parallel to $\mathcal L$.
Then there exists an exceptional orbit whose moment image lies in $\mathcal L$.
\end{lemma}

\begin{proof}
    By Lemma~\ref{lem:exceptional-orbits}, $[t,\nu,z]$ is exceptional if and only if $[t,0,z]$ is exceptional. Since $\mathcal{L}' \cap \mathfrak{h}^\circ =\{0\}$, we can assume without loss of generality that $\mathfrak{h}^\circ = \{0\}$.

    Fix $\alpha \in \mathcal{L}$, we show that there exist an index set $I$ with $|I| \leq h$ and $x_i > 0$ for each $i \in I$ such that $\alpha = \sum_{i \in I} x_i  \eta_i$ and $\{\eta_i: i \in I\}$ is linearly independent.
    Since $\alpha \in \Phi_Y(Y)$, there exist $y_0,\ldots,y_h \geq 0$ such that $\alpha = \sum_{i=0}^{h} y_i \eta_i$. Let $I^+:=\{i: y_i >0\}$. Without loss of generality, we can assume that $\{\eta_i: i\in I^+\}$ is linearly dependent. Let $\xi_0,\ldots,\xi_h \in \Z_{\geq 0}$ be the exponents of the defining monomial. Since $\{\eta_i: i\in I^+\}$ is linearly dependent, the last claim of Lemma~\ref{lem:defining-monomial} implies that $\xi_i =0$ for all $i \notin I^+$.
    Choose $s_0 <0$ such that $y_i+s_0\xi_i \geq 0$ for all $i=0,\ldots,h$ and $y_k + s_0\xi_k = 0$ for at least one $k \in I^+$. Notice that by construction $\xi_k >0$. Define $x_i:= y_i+s_0\xi_i$ and $I:=\{i: x_i >0\}$. By Lemma~\ref{lem:defining-monomial}, $\sum \xi_i \eta_i=0$, so $\alpha = \sum_{i=0}^h y_i \eta_i  =  \sum_{i=0}^h (y_i+s_0\xi_i) \eta_i  = \sum_{i=0}^h x_i \eta_i = \sum_{i \in I}  x_i \eta_i$. We show that $\{\eta_i: i \in I\}$ is linearly independent. Suppose $\sum_{i\in I} a_i \eta_i =0$. Then by the last claim of Lemma~\ref{lem:defining-monomial}, there exists a $\lambda \in \R$ such that $a_i = \lambda \xi_i$ for all $i=0,\ldots,h$. In particular, since $k \notin I$, $0=a_k = \lambda \xi_k$, so $\lambda =0$ and hence $\{\eta_i: i \in I\}$ is linearly independent.
    
    Since $\{\eta_i: i \in I\}$ is linearly independent, $\mathcal{L}$ is not contained in the proper cone $C_I:= \sum_{i \in I} \R_{\geq 0} \eta_i$, i.e. there exists a $\beta \in \mathcal{L} \setminus C_I$. Similarly, there exist an index set $J \not\subset I$ with $|J| \leq h$ and $w_j > 0$ for each $j \in J$ such that $\beta = \sum_{j \in J} w_j  \eta_j$ and $\{\eta_j: j \in J\}$ is linearly independent. Let $\gamma(t) = (1-t)\alpha + t\beta$, then there exists $t_0 \in(0,1)$ such that $\{t \in[0,1]: \gamma(t) \in C_I\} = [0,t_0]$. Hence, there exist $a_i \geq 0$ for each $i \in I$ and $i_0 \in I$ such that $a_{i_0}=0$ and $(1-t_0)\sum_{i \in I}x_i \eta_i +t_0 \sum_{j \in J}w_j \eta_j = \sum_{i \in I} a_i \eta_i$. By Lemma~\ref{lem:defining-monomial}, $\xi_j \neq 0$ for all $j \in J \setminus I$ and $\xi_{i_0} \neq 0$. Choose $z \in \C^{h+1}$ such that $z_i = \sqrt{a_i}$ for all $i\in I$ and $z_i =0$ for $i \notin I$. Now, $\sum_{z_j=0} \xi_j \geq \sum_{j \in J \setminus I} \xi_j+ \xi_{i_0}$. Since $\xi_j \neq 0$ for all $j \in J \setminus I$ and $\xi_{i_0} \neq 0$, $\sum_{z_j=0} \xi_j \geq 2$. By Lemma~\ref{lem:exceptional-orbits}, $[1,0,z]$ is an exceptional point and thus $\Phi_Y([1,0,z]) =\gamma(t_0) \in \mathcal{L}$ is the moment image of an exceptional orbit. 
\end{proof}

\begin{lemma}\label{lem:uppersemi}
Let $(M, \omega, \Phi)$ be a tall complexity one $T$-space with proper moment map and let $E$ be a nonempty clopen subset of the skeleton. The function
$$\alpha \mapsto \big|\Phi^{-1}(\alpha) \cap E \big|$$ is upper semi-continuous, where $\Phi: M/T \to \ft^*$ is the orbital moment map.
\end{lemma}

\begin{proof}
We first notice that by~\cite[Corollary 2.5]{KT14}, the number of exceptional orbits in any fiber is finite.
Given $\alpha \in \Delta$ and an orbit $o \in \Phi^{-1}(\alpha)$, we can identify a neighborhood $V_o$ of $o \subset M$ with an open subset of the local model $Y_o$ associated to $o$. By Lemma~\ref{lem:exceptional-orbits}, $E$ is a closed subset of $M/T$, so if $o \notin E$, we can shrink $V_o$ so that it contains no exceptional orbits in $E$. Together with Lemma~\ref{lem:2-colorable-local-model}, this implies that each moment fiber of $Y_o$ contains at most one exceptional orbit in $E$.

Since $\Phi$ is proper, there exist $o_1,\dots,o_k \in \Phi^{-1}(\alpha)$ and $0 \leq m \leq k$ such that $\Phi^{-1}(\alpha) \subseteq \cup_{i} V_{o_i}$ and such that $o_1,\ldots,o_m$ are exactly the exceptional orbits in $\Phi^{-1}(\alpha)\cap E$. Notice that the moment map $\Phi$ is closed, because it is proper and $\mathfrak t^*$ is locally compact and Hausdorff.
  Hence, there exists a neighborhood $U$ of $\alpha$ such that $\Phi^{-1}(U) \subseteq \cup_{i} V_{o_i}$. Now, for any point $\beta \in V$, $\Phi^{-1}(\beta) \cap V_{o_i}$ contains no exceptional orbits in $E$ if $i >m$ and at most one exceptional orbit in $E$ if $i \leq m$. Hence, $|\Phi^{-1}(\beta) \cap E| \leq |\Phi^{-1}(\alpha) \cap E|$.
\end{proof}

\begin{proposition} \label{prop:injectivity}
Let $(M,\omega, \Phi)$ be a compact connected tall complexity one $T$-space.
Let $E$ be a nonempty clopen subset of the skeleton of $M$ and let $\Gamma= \{\mathcal{O} \in E: \operatorname{dim} (\mathcal{O})\leq 1\}$. 
If the orbital moment map $\Phi: M/T \to \ft^*$ is injective on $\Gamma$, then it is injective on $E$.
\end{proposition}

\begin{proof}
We will prove the contrapostive.
Assume that $\Phi$ is not injective on $E$. Then $N = \{\alpha \in \Phi(M): |\Phi^{-1}(\alpha) \cap E| \geq 2\}$ is nonempty. By Lemma~\ref{lem:uppersemi}, $N$ is closed.

Pick any direction $\xi \in \mathfrak{t}$ 
which is generic in the following sense:
for any two points $p,q\in M$ such that $\Phi(p) = \Phi(q)$, if $\mathfrak{h}_p + \mathfrak{h}_q \neq \ft$ then $\xi\notin \mathfrak{h}_p + \mathfrak{h}_q$,
where $\mathfrak{h}_p$ and $\mathfrak{h}_q$ denote the Lie algebras of the stabilizer groups of $p$ and $q$.
Since $\Phi(M)$ is compact and $N \subset \Phi(M)$ is closed, $\langle \cdot, \xi \rangle$ attains its minimum in $N$ at some point $\alpha_{\min} \in N$.

Fix $\alpha \in N$ such that $\Phi^{-1}(\alpha)$ contains an exceptional orbit of dimension at least $2$. We will show that $\alpha \neq \alpha_{\min}$ by finding some $\alpha' \in N$ such that $\langle \alpha',\xi\rangle < \langle  \alpha, \xi \rangle$. This will imply that $|\Phi^{-1}(\alpha_{\min}) \cap \Gamma| = |\Phi^{-1}(\alpha_{\min}) \cap E| \geq 2$, and hence prove the claim.

By definition of $N$ and our assumption on $\alpha$, there exist distinct orbits $\mathcal O_p, \mathcal O_q \in \Phi^{-1}(\alpha) \cap E$ such that $\dim \mathcal O_p \geq 2$. 
Let $Y_p = T \times_{H_p}  \mathfrak{h}_p^{\circ} \times \C^{h_p+1},Y_q = T \times_{H_q}  \mathfrak{h}_q^{\circ} \times  \C^{h_q+1}$ be the local models associated to $p,q$ respectively with moment maps $\Phi_p,\Phi_q$, where $h_p = \dim H_p, h_q =\dim H_q$. Let $\Sigma_p \subset Y_p / T, \Sigma_q \subset Y_q/ T$ be the set of exceptional orbits in the local models of $p$ and $q$, respectively.
By Lemma~\ref{lem:exceptional-orbits}, since $\mathcal{O}_p$ is exceptional, the orbit $\mathcal{O}_{[1,\nu,0]} \in Y_p/T$ is exceptional for any $\nu\in\mathfrak{h}_p^\circ$. 
Hence, $\alpha+ \mathfrak{h}_p^{\circ} \subset \Phi_p(\Sigma_p)$, and 
similarly $\alpha+\mathfrak{h}_q^{\circ} \subset \Phi_q(\Sigma_q)$.

For $\epsilon >0$, we consider the hyperplane defined by 
$$\mathcal{H} =\{\eta \in \mathfrak{t}^*: \langle \eta, \xi \rangle = \langle \alpha, \xi \rangle - \epsilon\}.$$
First we will show that  $\Phi_p(\Sigma_p) \cap \Phi_q(\Sigma_q) \cap \mathcal{H} \neq \emptyset$.
There are two cases to consider.

If $\dim(\mathfrak{h}_p^{\circ} \cap \mathfrak{h}_q^{\circ})\geq 1$, then 
$\alpha+(\mathfrak{h}_p^{\circ} \cap \mathfrak{h}_q^{\circ})$ intersects $\mathcal{H}$ since $\xi$ is generic.
Since all points of $\alpha+(\mathfrak{h}_p^{\circ} \cap \mathfrak{h}_q^{\circ})$ are exceptional in both local models, we conclude that 
$\Phi_p(\Sigma_p) \cap \Phi_q(\Sigma_q) \cap \mathcal{H} \neq \emptyset$.

Otherwise, $\mathfrak{h}_p^{\circ} \cap \mathfrak{h}_q^{\circ} = \{0\}$. Since $\dim(\mathfrak{h}_p^{\circ})= \dim \mathcal{O}_p \geq 2$ and since $\xi$ is generic, there exists an affine line $\mathcal{L} \subset (\alpha+ \mathfrak{h}_p^{\circ} )\cap  \mathcal{H} $. 
By Lemma~\ref{lem:local-cone}, $\Phi_p(Y_p) = \Phi_q(Y_q)$. Hence, $\mathcal{L} \subset \Phi_q(Y_q)$ and $\mathcal{L}' \cap \mathfrak{h}_q^{\circ} = \{0\}$, where $\mathcal L'$ is the line through the origin which is parallel to $\mathcal L$. By Lemma~\ref{lem:affine-linelocalmodel}, we conclude that $\mathcal{L}\cap \Phi_q(\Sigma_q)\neq \emptyset$, and thus $\Phi_p(\Sigma_p) \cap \Phi_q(\Sigma_q) \cap \mathcal{H} \neq \emptyset$.

Let $\beta \in \Phi_p(\Sigma_p) \cap \Phi_q(\Sigma_q) \cap \mathcal{H}$. Then $\Phi_p^{-1}(\beta )\subset Y_p$ and $\Phi_q^{-1}(\beta)\subset Y_q$ both include exceptional orbits. By Lemma~\ref{lem:exceptional-orbits}, the interval $[\alpha,\beta]$ is a subset of $\Phi_p(\Sigma_p) \cap \Phi_q(\Sigma_q)$. 
By Lemma~\ref{lem:exceptional-saturated-new}, there exist neighborhoods $U\subset \Phi(M)$ of $\alpha$, $V_p \subset M/T$ of $\mathcal{O}_p$ and $V_q \subset M/T$ of $\mathcal{O}_q$ such that $E \cap \Phi^{-1}(U) \cap V_p$ is isomorphic to $\Sigma_p \cap \Phi_p^{-1}(U)$ and $E \cap \Phi^{-1}(U) \cap V_q$ is isomorphic to $\Sigma_q \cap \Phi_q^{-1}(U)$. 
Pick any $\alpha' \in U \cap (\alpha,\beta]$. Then $\Phi_p^{-1}(\alpha') \cap \Sigma_p \neq \emptyset$ and $\Phi_q^{-1}(\alpha') \cap \Sigma_q \neq \emptyset$. Hence, $\alpha' \in N$ and $\langle \alpha', \xi \rangle<\langle \alpha, \xi \rangle$, so $\alpha \neq \alpha_{\min}$.
\end{proof}

\begin{lemma}\label{lem:connectedness}
    Let $(M,\omega, \Phi)$ be a tall complexity one $T$-space with proper moment map. Let $E$ be a connected component of the skeleton of $M$ and let $\Gamma= \{\mathcal{O} \in E: \operatorname{dim} (\mathcal{O})\leq 1\}$. Then $\Gamma$ is connected.
\end{lemma}

\begin{proof}
    By Corollary~\ref{cor:exceptional-component}, $E = \cup (N_i/T)$ where each $N_i$ is a connected component of $M^{H_i}$ for some subgroup $H_i \subseteq T$. Hence, $E$ is path-connected. Let $p,q \in E$ be exceptional fixed points, then there exists a path $\gamma: [0,1] \to E$ such that $\gamma(0)=p, \gamma(1)=q$. We show that we can change $\gamma$ to make it lie in $\Gamma$. 
    
    First assume that $p,q \in N_i$. Since by Corollary~\ref{cor:exceptional-component} $N_i$ is a symplectic toric $(T/H_i)$-manifold, $N_i/T \cong \Phi(N_i)$ is a Delzant polytope and $\Phi(p),\Phi(q)$ are its vertices. We can connect $\Phi(p),\Phi(q)$ by a path consisting of edges of $\Phi(N_i)$, which corresponds to a path consisting of orbits of dimensions at most one (c.f. Delzant~\cite{De88}). Hence, $p,q$ can be connected by a path in $\Gamma$.

    Now suppose that $p,q$ are not in the same $N_i$ for any $i$. Without loss of generality, assume that $p \in N_1$. Then $\gamma$ intersects more than one $N_i/T$. Let $t_0 >0$ be the smallest number such that $\gamma(t_0) \in N_1 \cap N_i$ for some $i \neq 1$. Let $p_1 \in N_1 \cap N_i$ be an exceptional fixed point different from $p$. By the argument above, there exists a path in $\Gamma$ connecting $p$ with $p_1$. Concatenate this path with a path in $N_i/T$ starting at $p_1$ and ending at $\gamma(t_0)$ and further with the path $\gamma|_{[t_0,1]}$, we get a new path $\gamma_1$ from $p$ to $q$ such that $\gamma(t) \in \Gamma$ for $0 \leq t \leq \frac{1}{3}$. Inductively, we can modify the path to get a path $\gamma_0:[0,1] \to \Gamma$ such that $\gamma_0(0)=p$ and $\gamma_0(1)=q$.

    Finally, by Corollary~\ref{cor:exceptional-component}, every orbit of dimension one in $\Gamma$ can be connected to an exceptional fixed point, so $\Gamma$ is path-connected.
\end{proof}

Now, we are ready to prove Theorem~\ref{thm:k-colorability}.
\begin{proof}[Proof of Theorem~\ref{thm:k-colorability}]
Let $\Sigma$ be the skeleton of $(M,\omega,\Phi)$ and $\Gamma$ be the one-skeleton of $(M,\omega,\Phi)$. It is clear that if $\Sigma$ is $k$-colorable, then $\Gamma$ is $k$-colorable. It remains to prove the converse.

Suppose that $\Gamma = \Gamma_1 \sqcup \ldots\sqcup\Gamma_k$ for clopen subsets $\Gamma_1,\ldots,\Gamma_k \subseteq \Gamma$ such that the orbital moment map $\Phi: M/T \to \ft^*$ is injective on each $\Gamma_i$. Let $E_i$ be the smallest clopen (in the subspace topology on $\Sigma$) subsets of $\Sigma$ that contains $\Gamma_i$. We claim that $\Sigma = E_1 \sqcup \ldots \sqcup E_k$.

Fix an exceptional orbit $\mathcal{O}_p \in \Sigma$. Let $H$ be the stabilizer group of $p$ and let $N$ be the connected component of $M^H$ that contains $p$. By Corollary~\ref{cor:exceptional-component},  every point in $N$ is an exceptional point. In particular, the fixed points in $N$ are exceptional, so $(N/T) \cap \Gamma_i \neq \emptyset$ for some $i$. Since $N$ is connected, $\mathcal{O}_p\in N/T \subset E_i$. This implies that $\Sigma =E_1 \cup \ldots \cup E_k$.

Next, we show that $E_i \cap E_j = \emptyset$ whenever $i \neq j$. Suppose $E_i \cap E_j \neq \emptyset$ for some $i \neq j$. Then there exists a connected component $\mathcal{C}$ of $\Sigma$ such that $\mathcal{C} \subseteq E_i \cap E_j$. By Lemma~\ref{lem:connectedness}, $\mathcal{C} \cap \Gamma$ is connected, so $\mathcal{C} \cap \Gamma$ is contained in $\Gamma_s$ for a unique $s \in \{1,\ldots,k\}$. Then either $\Gamma_i \subset E_i \setminus \mathcal{C} \subsetneq E_i$ or $\Gamma_j \subset E_j \setminus \mathcal{C} \subsetneq E_j$, but this contradicts the minimality of $E_i$,$E_j$. Hence, $E_i \cap E_j = \emptyset$ whenever $i \neq j$. Therefore, $\Sigma = E_1 \sqcup \ldots \sqcup E_k$.

Now notice that $\Gamma_i = E_i \cap \Gamma$. Since $\Phi$ is injective on each $\Gamma_i$, Proposition~\ref{prop:injectivity} implies that $\Phi$ is injective on each $E_i$. We conclude that $\Sigma$ is $k$-colorable. 
\end{proof}

\section{Complexity one properties of a symplectic toric manifold}\label{sec:cx-1-toric}
In this section, we focus on a special family of complexity one spaces, the ones that are obtained from restricting a toric action on a symplectic toric manifold to a codimension one sub-torus action. We start with a few definitions.

Let $\ft$ be the Lie algebra of a compact torus $T$ and $\ell \subset \ft$ be its integral lattice. A subset $P$ of $\ft^*$ is {\bf polyhedral} if it is the intersection of finitely many closed half-spaces.
A polyhedral set is {\bf rational} if each half-space admits an outwards pointing normal vector which is primitive in $\ell$.
Additionally, a rational  polyhedral set $P$ is {\bf unimodular} if for each point $\alpha\in P$ the primitive outward normal vectors of the half-spaces that $\alpha$ lies on the boundary of can be extended to a $\Z$-basis of $\ell$. 
$P$ is  {\bf locally unimodular} if for any point $\alpha \in P$ there is a neighborhood $U_\alpha$ of $\alpha$ in $\ft^*$ and a unimodular rational polyhedral set $D_\alpha$ such that $U_\alpha \cap P = U_\alpha \cap D_\alpha$.  In particular, every open set is locally unimodular. Finally, a {\bf Delzant polyhedral set} is a unimodular, rational polyhedral set and a {\bf Delzant polytope} is a compact Delzant polyhedral set. Delzant~\cite{De88} classified compact connected symplectic toric manifolds by their moment images.

\begin{theorem}[\cite{De88}]\label{thm:Delzant}
There is a one-to-one correspondence.
\[\{\text{compact, connceted symplectic toric $T$-manifolds}\}/_{\sim} \xlongleftrightarrow{1-1} \{\text{Delzant polytopes in $\ft^*$}\}/_{\sim}\]
The equivalence relation is given by equivariant symplectomorhisms that intertwine the moment maps on the left and by translations on the right.
\end{theorem}

In non-compact case, the moment image is locally unimodular instead of unimodular, see~\cite[Theorem 4.3]{LMTW} and~ \cite[Section 6]{KL}. 
Karshon and Lerman~\cite{KL} classified non-compact symplectic toric manifolds. We recall part of their classification results and refer the readers to~\cite[Theorem 1.3]{KL} for the complete story. To state it, we need the following definition from \cite[Definition A.16]{KL}.

\begin{definition}
Let a compact torus $T$ act on a manifold $M$. We say that a smooth map $\pi$ from
$M$ to a manifold with corners $W$ is a $\mathbf{T}${\bf-quotient map} if for every $T$-invariant smooth function $f \colon  M \to Y$
there exists a unique smooth map $\overline{f} \colon W \to Y$  such that $f = \overline{f} \circ \pi$ and, conversely, if for every smooth map $\overline{f} \colon W \to Y$ the pull-back $f := \overline{f} \circ \pi$ is a  $T$-invariant smooth function.\footnote{In \cite{KL}, the authors inadvertently omit the assumption that $\overline{f} \circ \pi$ is $T$-invariant from~\cite[Definition A.16]{KL}, but this is implied by their assertion that  $\pi$ identifies $W$ with $M/T.$}
Such a map induces a homeomorphism from the space $M/ T$ of $T$-orbits to $W$.
\end{definition}

\begin{theorem}[\cite{KL}]\label{thm:KL}
    Let $\Delta \subseteq  \ft^*$ be a locally unimodular set.
There exists a symplectic toric $T$-manifold $(M, \omega, \Phi)$
such that $\Phi(M) = \Delta$ and $\Phi \colon  M \to \Delta$ is a $T$-quotient map.
Moreover, if $\Delta$ is convex then $( M, \omega, \Phi)$ is unique up to isomorphisms  ($T$-equivariant symplectomorphisms that intertwine the moment maps).
\end{theorem}

The moment image of a tall complexity one space is also locally unimodular.
\begin{lemma}[Lemma 7.3 in~\cite{KT14}]\label{lem:cx-1-image-Delzant}
    Let $(M,\omega, \Phi)$ be a tall complexity one $T$-space with proper moment map. Then $\Phi(M)$ is a Delzant polyhedral set.
\end{lemma}

For the rest of this paper, we will mainly focus on locally unimodular sets over a Delzant polytope. We will use $\pi$ to denote the natural projection $\ft^* \times \R \to \ft^*$. We need the following definition.

\begin{definition}\label{def:top-bottom}
    Given  a subset $\Tilde{U}$ of $\ft^* \times \R$, let
\begin{align*} \del_+ \Tilde{U}  &:= \{(\alpha,x)\in\Tilde{U}\mid \text{if }(\alpha,x')\in\Tilde{U}\textrm{ then } x \geq x'\}\\
\del_- \Tilde{U}  &:= \{(\alpha,x) \in \Tilde{U} \mid \text{if }(\alpha,x')\in \Tilde{U} \textrm{ then } x \leq x' \}. \end{align*}
We call $\del_+ \Tilde{U} $ the {\bf top} of $\Tilde{U}$ and $\del_- \Tilde{U} $ the {\bf bottom} of $\Tilde{U}$.
\end{definition}

The following two lemmas give a criterion for when a point $p$ in the symplectic toric $(T\times S^1)$-manifold $(M,\omega,\Tilde{\Phi})$ is not exceptional, considered as a point in the complexity one $T$-space $(M,\omega,\pi \circ \Tilde{\Phi})$. 

\begin{lemma}[\cite{LPT}]\label{lem:equal-orbits-slice-rep}
    Let $(M,\omega,\Tilde{\Phi})$ be a symplectic toric $(T \times S^1)$-manifold such that $\Tilde{\Phi}(M)$ is convex and $\Tilde{\Phi} \colon M \to \Tilde{\Phi}(M)$ is a $(T \times S^1)$-quotient map. 
  Fix a tall point $p$ in the complexity one $T$-space $(M,\omega,\pi \circ \Tilde{\Phi})$.
    \begin{enumerate}
        \item[(1)]  If $\Tilde{\Phi}(p)$ is on the bottom or top of $\Tilde{\Phi}(M)$ then
    the orbits $T \cdot p$ and $(T \times S^1) \cdot p$ are identical.
      \item[(2)] If $\Tilde{\Phi}(p)$ is not on the bottom or top of $\Tilde{\Phi}(M)$ then $p$ is non-exceptional.
     \end{enumerate}
     Consequently, if $p$ is exceptional, then $\Tilde{\Phi}(p)$ is on the bottom or top of $\Tilde{\Phi}(M)$.
\end{lemma}

\begin{lemma}\label{lem:flat-top-nonexception}
    Let $\De \subset \ft^*$ be a Delzant polytope. Fix a continuous function $f: \De \to \R$ and $c > \max\{f(\alpha): \alpha \in \De\}$ such that $\Tilde{\De} := \{(\alpha,x) \in \De \times \R: c \geq x \geq f(\alpha)\}$ is a Delzant polytope. Let $(\Tilde{M},\Tilde{\omega},\Tilde{\Phi})$ be a symplectic toric $(T\times S^1)$-manifold with moment image $\Tilde{\De}$. Then any $p \in \Tilde\Phi^{-1}(\del_+\Tilde{\De})$ is not an exceptional point of the complexity one $T$-space $(\Tilde{M},\Tilde{\omega},\pi \circ \Tilde{\Phi})$.
\end{lemma}

\begin{proof}
    Fix $\epsilon>0$ such that $c-\epsilon > \max\{f(\alpha): \alpha \in \De\}$. Let $\Tilde{{\De}}_\epsilon = \De \times [c-\epsilon,c]$ and let $(\Tilde{M}_\epsilon,\Tilde{\omega}_\epsilon,\Tilde{\Phi}_\epsilon)$ be a symplectic toric $(T\times S^1)$-manifold with moment image $\Tilde{\De}_\epsilon$. It is straightforward to verify that the complexity one $T$-space $(\Tilde{M}_\epsilon,\Tilde{\omega}_\epsilon,\pi \circ \Tilde{\Phi}_\epsilon)$ has no exceptional points. By Theorem~\ref{thm:KL}, since $\Tilde{\De}_\epsilon$ is convex, $(\Tilde{M}_\epsilon,\Tilde{\omega}_\epsilon,\Tilde{\Phi}_\epsilon)$ is isomorphic to $(\Tilde{\Phi}^{-1}(\Tilde{\De}_\epsilon),\Tilde{\omega},\Tilde{\Phi})$. In particular, $\Tilde\Phi^{-1}(\del_+\Tilde{\De}) \subset \Tilde{\Phi}^{-1}(\Tilde{\De}_\epsilon)$, so the claim follows.
\end{proof}

We end this section with a technical lemma from~\cite{LPT} that we will need later in proving Theorem~\ref{thm:toric}.

\begin{lemma}[\cite{LPT}]\label{lem:slice-rep}
    Let $(M,\omega,\Tilde{\Phi})$ be a symplectic toric $(T \times S^1)$-manifold such that  $\Tilde{\Phi}(M)$ is convex and  $\Tilde{\Phi}  \colon  M \to \Tilde{\Phi}(M)$ 
     is a $(T \times S^1)$-quotient map. Fix a tall point $p$ in the complexity one $T$-space $(M,\omega,\pi \circ \Tilde{\Phi})$ such that $\Tilde{\Phi}(p) \in \del_- \Tilde{\Phi}(M)$ .
       Let  $\Y$ be the local model for $p$ in $(M,\omega,\pi \circ \Tilde{\Phi})$ with defining monomial $P$. Let $\Tilde{P}: T\times_H (S^1)^{h+1} \to S^1$ be defined by $\Tilde{P}([t,z]) = P(z)$.
 Then a  $(T\times S^1)$-invariant open neighborhood of $p$ is
        $(T\times S^1)$-equivariantly symplectomorphic to an open neighborhood of $[1,0,0]$
        in $Y$, where $T \times S^1$ acts on $Y$ by
      a $T$-equivariant isomorphism $\Theta \colon  T\times S^1 \to T \times_H (S^1)^{h+1}$ such that  $\Tilde P(\Theta(t,\lambda)) = \lambda$ for all $(t, \lambda) \in T \times S^1$.
\end{lemma}

\section{Proof of Theorem~\ref{thm:toric}}\label{sec:extension}
In this section, we prove Theorem~\ref{thm:toric}. The main argument works very similarly to the argument in~\cite{LPT}.
The following local existence result in~\cite{LPT} is key to the local existence in our case.

\begin{lemma}[\cite{LPT}]\label{lem:raystrong}
Let $Y = T \times_H \mathfrak{h}^\circ \times \C^{h+1}$ be a tall complexity one local model with moment map $\Phi_Y  \colon Y \to \ft^*$. Then there exists a symplectic toric $(T\times S^1)$-action extending the original $T$-action with moment map $\Tilde{\Phi}_Y: Y \to \ft^* \times \R$ such that $\Tilde{\Phi}_Y$ is a $(T\times S^1)$-quotient map and 
$\Tilde \Phi_Y( Y) = \{(\alpha, x) \in \Phi_Y(Y) \times \R \mid x \geq f(\alpha) \}$ for some continuous function $f: \Phi_Y(Y) \to \R$.
Moreover, $\Tilde{\Phi}_Y(Y)$ is a Delzant polyhedral set.
\end{lemma}

The following proposition is a variation of \cite[Proposition 5.5]{LPT} and takes care of the local uniqueness of our construction. 
\begin{proposition}\label{prop:local-uniqueness}
Fix a convex set $U \subseteq \ft^*$ and a continuous function $f_i \colon U \to \R$ for $i \in \{1,2\}$ such that that the following hold:
\begin{enumerate}
    \item  The set $\Tilde U_i:=\{ (\alpha,x)  \in  U \times \R \mid  x \geq f_i(\alpha)\}$ is convex and locally unimodular for $i \in \{1,2\}.$
    \item There exist symplectic toric $(T \times S^1)$-manifolds  $(\Tilde M_1, \Tilde \omega_1, \Tilde \Phi_1)$ and $(\Tilde M_2, \Tilde \omega_2, \Tilde \Phi_2)$   and  an isomorphism from  the skeleton of the tall complexity one $T$-space $(\Tilde M_1,\Tilde \omega_1, \pi \circ \Tilde \Phi_1)$ to the skeleton of $(\Tilde M_2, \Tilde \omega_2, \pi \circ \Tilde \Phi_2)$; moreover, $\Tilde \Phi_i \colon \Tilde M_i \to \Tilde U_i$ is a $(T \times S^1)$-quotient map for $i \in \{1,2\}.$   
    \end{enumerate}
  Then there exist $A \in \ell$ and $B \in \R$ such that $f_2(\alpha)-f_1(\alpha) = \langle \alpha ,A \rangle +B$ for all $\alpha \in U$.
\end{proposition}

\begin{proof}
    Fix $\beta \in U$ and $p_i \in \Tilde M_i$ such that
$\Tilde \Phi_i(p_i) = (\beta, f_i(\beta))$ for $i \in \{1,2\}$.
Our first claim is that $p_1$ and $p_2$ have the same $T$-stabilizer and $T$-slice representation.
To see this, note that since $\Tilde \Phi_i$ induces
a homeomorphism from $\Tilde M_i/(T \times S^1)$ to $\Tilde U_i$,
there exists a unique $(T \times S^1)$-orbit $o_i$ in $\Tilde M_i$ such that
$\Tilde \Phi_i(o_i) = (\beta, f_i(\beta))$ for $i \in \{1,2\}$; by part (1) of Lemma~\ref{lem:equal-orbits-slice-rep} $o_i$ is also the $T$-orbit through $p_i$. By assumption, $(\Tilde M_1,\Tilde \omega_1, \pi \circ \Tilde \Phi_1)$ and $(\Tilde M_2, \Tilde \omega_2, \pi \circ \Tilde \Phi_2)$ have isomorphic skeletons. Moreover, by the last claim of Lemma~\ref{lem:equal-orbits-slice-rep}, each exceptional orbit is mapped to $\del_- \Tilde{U}_i$.
Hence, if either $p_1$ or $p_2$ is  exceptional, then both are exceptional and they have the same $T$-stabilizer and $T$-slice representation.  Otherwise,  neither $p_1$ nor $p_2$ is exceptional, in which case $p_1$ and $p_2$ still have the same $T$-stabilizer and $T$-slice representation, because they are determined by the local information around $\beta \in U$.

Let $Y = T \times_H \fh^\circ \times \C^{h+1}$ be the local model for both $p_1$ and $p_2$. Let
$P$ be the defining monomial and define $\Tilde{P} \colon T \times_H (S^1)^{h+1} \to S^1$ to be $\Tilde{P}([t,z]) = P(z)$. By Lemma~\ref{lem:slice-rep},
a $(T \times S^1)$-invariant open neighborhood  $V_i$ of $p_i$ in $M_i$ is
        $(T \times S^1)$-equivariantly symplectomorphic to an open neighborhood  of $\{[t,0,0]\}$
        in $Y$, where $T \times S^1$ acts on $Y$ by
      a $T$-equivariant isomorphism $\Theta_i \colon T \times S^1 \to T \times_H (S^1)^{h+1}$ that satisfies  $\Tilde P(\Theta_i(t,\lambda)) = \lambda$ for all $(t, \lambda) \in T \times S^1.$ 
Since  $\Theta_1(1, \lambda) \cdot  \Theta_2(1,\lambda)^{-1}$
lies in $\ker \Tilde P =T \times_H H \simeq T$ for all $\lambda \in S^1$,
there exists a homomorphism $A \colon S^1 \to T$ such that
$\Theta_2(1, \lambda) = A(\lambda) \cdot \Theta_1(1, \lambda)$ for all $\lambda \in S^1$.   Since $\Theta_i$ is $T$-equivariant, 
$\Theta_2(t, \lambda) =\Theta_1(A(\lambda) \, t,  \lambda)$ for all $(t, \lambda) \in T \times S^1$. 
By a slight abuse of notation, also call the element of the integral lattice $\ell \subset \ft$
associated to $A \colon S^1 \to T$ by $A$.

Let $\Tilde \Psi_1, \Tilde \Psi_2 \colon Y \to \ft^* \times \R$ be the moment maps for the $(T \times S^1)$-action on $Y$ induced by the isomorphisms $\Theta_1$ and $\Theta_2$, respectively. 
Since $\Theta_2$ is the composition of $\Theta_1$ with the automorphism of $T \times S^1$
that sends $(t, \lambda)$ to $( A(\lambda)\, t,  \lambda)$,
the moment polytope $\Tilde\Psi_2(Y)$ is the image of $\Tilde\Psi_1(Y)$ under
the automorphism of $\ft^* \times \R$ given by 
$(\alpha, x) \mapsto  \big(\alpha, x + \langle \alpha - \beta, A \rangle + f_2(\beta)  -f_1(\beta) \big).$
Therefore, after possibly shrinking $V_1$ and $V_2$, 
$\Tilde \Phi_2(V_2)$ is the image of $\Tilde \Phi_1(V_1)$ under the same automorphism.
Since $\Tilde \Phi_i \colon M_i/(T \times S^1) \to \Tilde U_i$ is a homeomorphism for $i \in \{1,2\},$
this implies that 
$$f_2(\alpha) - f_1(\alpha)   =  \langle  \alpha -\beta, A \rangle  + f_2(\beta) - f_1(\beta)$$ for all $\alpha$ in a neighborhood of $\beta$.
Therefore, $f_2 - f_1$ is locally integral affine on $U$.
Since $U$ is convex, $U$ is connected, so $f_2 - f_1$ is integral affine on $U$.
\end{proof}

Now, we are ready to prove Theorem~\ref{thm:toric}. The following theorem plays an important role in both the proof of Theorem~\ref{thm:toric} and Section~\ref{sec:graph}.

\begin{theorem} \label{thm:inj-toric-alt}
Let $(M,\omega,\Phi)$ be a connected tall complexity one $T$-space with proper moment map and moment polytope $\De:= \Phi(M)$. Let $E$ be a clopen (closed and open) subset of the skeleton such that the orbital moment map $\Phi \colon E \to \ft^*$ is one-to-one. Then there exists a continuous function $f: \De \to \R$ such that 
\begin{enumerate}
\item $\Tilde\De:= \{(\alpha,x) \in \De \times \R: x \geq f(\alpha)\}$ is a Delzant polyhedral set.
\item There exists a connected symplectic toric $(T \times S^1)$-manifold  $(\Tilde M, \Tilde \omega, \Tilde \Phi)$ such that $\Tilde \Phi: M \to \Tilde \De$ is a $(T\times S^1)$-quotient map and the skeleton of the tall complexity one $T$-space $(\Tilde M, \Tilde \omega, \pi \circ \Tilde \Phi)$ is isomorphic to $E$.
\end{enumerate}
\end{theorem}

\begin{proof}
Define a presheaf $\mathcal P$ on $\Delta$  as follows.  Given an open subset $U \subseteq \Delta$, let $\mathcal P(U)$ be the set of  continuous functions $f\colon U \to \R$ such that the following hold:
\begin{enumerate}
    \item The set $\Tilde U=\{ (\alpha,x)  \in  U \times \R \mid  x \geq f(\alpha) \}$ is locally unimodular.
    \item There exists a symplectic toric $(T \times S^1)$-manifold  $(\Tilde M, \Tilde \omega, \Tilde \Phi)$ such that $\Tilde \Phi: \Tilde M \to \Tilde U$ is a $(T \times S^1)$-quotient map and $E \cap \Phi^{-1}(U)$ is isomorphic to the skeleton of the tall complexity one $T$-space $(\Tilde M, \Tilde \omega, \pi \circ \Tilde \Phi)$.
\end{enumerate}
     
If $f \in \mathcal P(U)$ and if $U' \subseteq U$ is an open subset,
 then clearly $f|_{U'} \in \mathcal P(U')$. 
Furthermore, if $f_1 \in \mathcal P(U)$,  $A \in \ell \subset \ft$, and $B \in \R$, then the function
$f_2 := f_1 + \langle \cdot, A\rangle +B$
is still continuous.
Then $\Tilde{U}_2$ is the image of $\Tilde{U}_1$ under the affine transformation of $\ft^* \times \R$ given by $(\alpha, x) \mapsto (\alpha, x + \langle \alpha, A \rangle + B)$, so $\Tilde U_2$ is locally unimodular. Define $\Tilde{\Phi}_2$ to be the composition of $\Tilde{\Phi}_1$ and this affine transformation. Then $(\Tilde{M}_1, \Tilde{\omega}_1, \Tilde{\Phi}_2)$ is a symplectic toric $(T \times S^1)$-manifold such that $\Tilde \Phi_2: \Tilde{M}_1 \to \Tilde{U}_2$ is a $(T \times S^1)$-quotient map and the tall complexity one $T$-spaces $(\Tilde M_1, \Tilde\omega_1, \pi \circ \Tilde \Phi_1)$ and $(\Tilde M_1, \Tilde\omega_1, \pi \circ \Tilde \Phi_2)$ are identical. We conclude that $f_2 \in \mathcal{P}(U)$ as well.

We next show that there exists a cover $\mathfrak U$ of $\Delta$ by convex open sets
such that $\mathcal P(U) \neq \emptyset$ for all $U \in \mathfrak U$. 
Fix $\beta \in \De$. 
Pick an orbit $o \in \Phi^{-1}(\beta)$ such that $\Phi^{-1}(\beta) \cap E = \{o\}$ if $\beta \in \Phi(E)$ and $o$ is non-exceptional otherwise.
Let $Y= T \times_H \mathfrak{h}^\circ \times \C^{h+1}$ be the local model associated to $o$; 
normalize the moment map $\Phi_Y \colon Y \to \ft^*$ so that $\Phi_Y([1,0,0]) = \beta.$
By Lemma~\ref{lem:exceptional-saturated-new}, there exist a convex open neighborhood $U$ of $\beta$ in $\De$ and a neighborhood $V_o$ of $o$ in $M/T$ such that $E \cap \Phi^{-1}(U) \cap V_o$ is isomorphic to the set of exceptional orbits in $\Phi_Y^{-1}(U)$. Since $\Phi$ is injective on $E$ and $\Phi|_E$ is a closed map, $\Phi$ is a homeomorphism from $E$ to $\Phi(E)$, so after possibly shrinking $U$, we may assume that $E \cap \Phi^{-1}(U) \cap V_o= E \cap \Phi^{-1}(U)$. Hence, $E \cap \Phi^{-1}(U)$ is isomorphic to the set of exceptional orbits in $\Phi_Y^{-1}(U)$.

By Lemma~\ref{lem:raystrong}, there exists a symplectic toric $(T\times S^1)$-action on $Y$ extending the original $T$-action with moment map $\Tilde{\Phi}_Y: Y \to \ft^* \times \R$ such that $\Tilde{\Phi}_Y$ is a $(T\times S^1)$-quotient map and 
$\Tilde \Phi_Y( Y) = \{(\alpha, x) \in \Phi_Y(Y) \times \R \mid x \geq f(\alpha) \}$ for some continuous function $f: \Phi_Y(Y) \to \R$.
Furthermore, $\Tilde \Phi_Y(Y)$ is a Delzant polyhedral set.
We prove that $f|_U \in \mathcal{P}(U)$.

(1) Define $\Tilde U := \{ (\alpha,x)  \in  U \times \R \mid  x \geq f(\alpha) \}$.
Notice that $\Tilde U$ is the intersection of the  set $U \times \R$ with
$\Tilde\Phi_Y(Y)$.  Since both sets are convex, this implies that $\Tilde U$ is convex. Moreover, since $\Tilde \Phi_Y(Y)$ is a locally unimodular polyhedral set, $\Tilde U$ is also locally unimodular.  

(2) By Theorem~\ref{thm:KL}, since $\Tilde U$ is convex, $\Tilde M:= \Tilde \Phi_Y^{-1}(\Tilde U)$ is the unique symplectic toric $(T \times S^1)$-manifold with moment image $\Tilde U$. Hence, the skeleton of the tall complexity one $T$-space $(\Tilde M,\Tilde \omega, \pi \circ \Tilde\Phi_Y)=(\Phi_Y^{-1}(U),\Tilde \omega, \Phi_Y)$ is isomorphic to $E \cap \Phi^{-1}(U)$. 

Therefore, there exist a cover $\mathfrak U$ of $\Delta$ by convex open sets and function $f_U \in \mathcal{P}(U)$ for each $U \in \mathfrak{U}$ such that  the set $\Tilde U$ defined in (1) is convex. Consider $U, V \in \mathfrak U$
such that $U \cap V \neq \emptyset.$
By Proposition~\ref{prop:local-uniqueness}, since $U \cap V$ is
convex,
there exist unique $A_{UV} \in \ell$ and $B_{UV} \in \R$ such that 
$$f_V(\alpha) =  f_U(\alpha)  + \langle \alpha,  A_{UV} \rangle + B_{UV} \quad \text{for all }
\alpha \in U \cap V.
$$ 
Hence, given $U, V, W \in \mathfrak U$ such that $U \cap V \cap W \neq \emptyset$, the following hold for all $\alpha \in U \cap V \cap W:$
\begin{gather*}
    f_W(\alpha) = f_U(\alpha) + \langle \alpha, A_{UW} \rangle + B_{UW}, \quad \text{and} \\
f_W(\alpha) =f_V(\alpha)  + \langle \alpha, A_{VW} \rangle + B_{VW}  = f_U(\alpha) + \langle \alpha, A_{UV} + A_{VW} \rangle + B_{UV} + B_{VW}.
\end{gather*}Therefore,  $A_{UW} = A_{UV} + A_{VW}$ and $B_{UW} = B_{UV} + B_{VW}$, and so  $\{(A_{UV}, B_{UV})\}_{U, V \in \mathfrak U}$ is a \v{C}ech cocyle in $\check C^1(\mathfrak U, \ell \times \R)$.

Since $\Delta$ is contractible, $\check H^1(\Delta, \ell \times \R) \cong H^1(\Delta, \ell \times \R) = 0$.
Therefore, by passing to a refinement, we may assume that    $[(A_{UV}, B_{UV})] = 0  \in \check H^1(\mathfrak U, \ell \times \R)$.  
This means that  there exists $(a_U, b_U) \in \ell \times \R$ for each $U \in \mathfrak U$ such that
$A_{UV} = a_V - a_U$ and $B_{UV} = b_V - b_U$ for all $U, V \in \mathfrak U.$
Given $U \in \mathfrak U$, by the argument above, we may replace $f_U$ by 
$$f_U - \langle \cdot, a_U \rangle -b_U$$
and the new functions, which we still denote by $f_U$, are still in $\mathcal P(U)$.
Having done this, $f_U = f_V$ on overlaps on $U \cap V$ for all $U, V \in \mathfrak U.$
Therefore, they define a global function $f \colon \Delta \to \R$.

It is clear that $f$ is continuous, and that $\Tilde \De := \{ (\alpha,x) \mid \alpha \in \Delta \text{ and }  x \geq f(\alpha) \}$ is a Delzant polyhedral set. Therefore, by Theorem~\ref{thm:KL}, there exists a symplectic toric $(T \times S^1)$-manifold $(\Tilde M, \Tilde \omega, \Tilde \Phi)$ such that $\Tilde \Phi$ is a $(T \times S^1)$-quotient map. Moreover, since $\Tilde \De$ is a polyhedral set, it is connected, so $\Tilde M$ is connected. For any $p \in \Tilde M$, $(\pi\circ\Phi)^{-1}(\Phi(p)) = \Tilde{\Phi}^{-1}(\Phi(p) \times \R)$. Since $\Tilde{\Phi}$ induces a homeomorphism from $\Tilde{M}/(T\times S^1)$ to $\Tilde{\De}$ and $(\Phi(p) \times \R) \cap \Tilde{\De}$ contains more than one point, $(\pi\circ\Phi)^{-1}(\Phi(p))$ contains more than one $T$-orbit. Hence, $(\Tilde M, \Tilde \om, \pi\circ \Tilde \Phi)$ is a tall complexity one $T$-space.
By construction, its skeleton is isomorphic to $E$.
\end{proof}

\subsection*{Proof of Theorem~\ref{thm:toric}}
    By Theorem~\ref{thm:inj-toric-alt}, there exists a continuous function $f: \De \to \R$ such that 
\begin{enumerate}
\item[(1)] $\Tilde{\De}:= \{(\alpha,x) \in \De \times \R: x \geq f(\alpha)\}$ is a Delzant polyhedral set.
\item[(2)] There exists a connected symplectic toric $(T \times S^1)$-manifold  $(\Tilde{M}, \Tilde{\omega}, \Tilde{\Phi})$ such that $\Tilde{\Phi} :\Tilde{M} \to \Tilde{\De}$ is a $(T \times S^1)$-quotient map, and the skeleton of the tall complexity one $T$-space $(\Tilde{M}, \Tilde{\omega}, \pi \circ \Tilde{\Phi})$ is isomorphic to $E$.
\end{enumerate}
Take $c > \max \{f_i(\alpha): \alpha \in \De\}$ and define $\Tilde{\De}':= \{(\alpha,x) \in \De \times \R:c\geq  x \geq f(\alpha)\}$. We show that $\Tilde{\De}'$ is a Delzant polytope. It is clear that $\Tilde{\De}'$ is a simple rational polytope and $\Tilde{\De}' \setminus \partial_+ \Tilde{\De}'$ is locally unimodular. Moreover, by Lemma~\ref{lem:cx-1-image-Delzant}, $\De$ is a Delzant polytope, so $\De \times I$ is still a Delzant polytope for any closed interval $I \subset \R$. Pick $\epsilon >0$ such that $c-\epsilon >  \max \{f_i(\alpha): \alpha \in \De\}$. Then $\Tilde{\De}' = (\De \times [c-\epsilon,c]) \cup (\Tilde{\De}' \setminus \partial_+ \Tilde{\De}')$ is locally unimodular, and hence Delzant.
 
By Theorem~\ref{thm:Delzant}, there exists a compact, connected symplectic toric $(T \times S^1)$-manifold  $(M', \omega', \Psi)$ with moment polytope $\Tilde\De'$. Set $\Phi':= \pi \circ \Psi$. We show the connected complexity one space $(M', \omega', \Phi')$ satisfies the requirements. By construction, $\Phi(M) = \De = \Phi'(M')$ and the $T$-action on $M'$ extends to a toric action. Since $\Tilde{\De}' \cap (\{\alpha\} \times \R)$ is an interval for each $\alpha \in \De$, and since $\Tilde{\De'}$ is homeomorphic to $M'/(T \times S^1)$, the complexity one space is tall. 
Since $\Tilde{\De}' \setminus \partial_+ \Tilde{\De}'$ is convex, Theorem~\ref{thm:KL} implies that $\Tilde{\Phi}^{-1}(\Tilde{\De}' \setminus \partial_+ \Tilde{\De}')$ is isomorphic to $\Psi^{-1}(\Tilde{\De}' \setminus \partial_+ \Tilde{\De}')$. Hence, $E$ is isomorphic to the set of exceptional $T$-orbits in $\Psi^{-1}(\Tilde{\De}' \setminus \partial_+ \Tilde{\De}')$. By Lemma~\ref{lem:flat-top-nonexception}, the skeleton of $(M', \omega', \Phi')$ is a subset of $\Psi^{-1}(\Tilde{\De}' \setminus \partial_+ \Tilde{\De}')$. Therefore, $E$ is isomorphic to the skeleton of $(M', \omega', \Phi')$.

\section{One-skeleton determines the skeleton}\label{sec:graph}

In this section, we prove Theorem~\ref{thm:1-skeleton}. We will use Theorem~\ref{thm:inj-toric-alt} to reduce the case to where both complexity one $T$-spaces are restrictions of a symplectic toric manifold. Hence, we will focus on the following scenario.

Let $\Delta \subset \ft^*$ be a Delzant polytope and let $\V(\De)$ be the set of vertices of $\De$. Let $f: \De \to \R$ be a continuous function whose epigraph is a Delzant polyhedral set $\widetilde{\De}:= \{(\alpha,x) \in \De \times \R: x \geq f(\alpha)\}$. Let $\V(\widetilde{\De})$ be the set of vertices of $\widetilde{\De}$. The {\bf one-skeleton of} $(\Delta,f)$ is a labeled graph defined in the following way:
\begin{itemize}
    \item the vertex set $\V$ is the set $\V(\widetilde{\De}) \setminus (\V(\De)\times \R)$, viewed as a subset of $\ft^* \times \R$;
    \item the set $\E$ of directed edges of the graph is the set of ordered pairs $(v,w) \in \V \times \V$ such that $[v,w]$ is an edge of $\widetilde{\De}$; and
    \item to each vertex $v \in \V$, we associate a weight label, which is a set $W_v$ of primitive vectors in $\ell^* \oplus \Z$ such that $\widetilde{\De} \cap U= (v+\sum_{\eta \in W_v} \R_{\geq 0} \eta) \cap U$ for some neighborhood $U$ of $v \in \widetilde{\De}$.
\end{itemize}

Notice that for each $e=(v,w) \in \E$, there exists a unique $\eta(e) \in W_v$ such that $w-v \in \R_{>0} \eta(e)$. This induces an inclusion $\eta :\E_v \to W_v$, where $\E_v$ is the set of edges emanating from $v$. We denote the one-skeleton of $(\De,f)$ by $\Gamma(\De,f) = (\V(\De,f),\E(\De,f),W(\De,f))$. 

\begin{lemma}\label{lem:mod-equal}
    Let $\De \subset \ft^*$ be a Delzant polytope and let $f: \De \to \R$ be a continuous function whose epigraph is a Delzant polyhedral set $\Tilde{\De}$. Let $\Gamma(\De,f) = (\V(\De,f),\E(\De,f),W(\De,f))$ be the one-skeleton of $(\De,f)$. For any edge $e=(v,w) \in \E(\De,f)$ and any $\eta \in W_v$, there exists a unique $\eta' \in W_w$ such that $\eta' -\eta \in\Z \eta(e)$, where $\eta(e) \in W_v$ is the weight directing the edge from $v$ to $w$.
\end{lemma}
\begin{proof}
    First notice that $-\eta(e) \in W_w$. Hence, if such a $\eta'$ exsits, then it must be unique, as $W_w$ is a basis of $\ell^* \oplus \Z$.
    If $\eta = \eta(e)$, then $\eta' = -\eta(e)$ is the unique weight in $W_w$ satisfying the condition. Otherwise, consider the $2$-dimensional face $F$ of $\Tilde{\De}$ containing $v,w$ such that $\eta,\eta(e) \in \R(F-v)$. By definition, there exists an open neighborhood $U$ around $w$ and a unique $\eta' \in W_w$ such that $F \cap U =( w+ \R_{\geq 0} \eta'- \R_{\geq 0} \eta(e)) \cap U$. Since $W_v,W_w$ are both basis of $\ell^* \oplus \Z$ and $F$ is a polygon, $\eta,\eta(e)$ and $\eta',-\eta(e)$ span the same lattice. The claim follows.
\end{proof}

\begin{definition}\label{def:isom-1-skeleton}
For $i=1,2$, let $\De_i\subset \ft^*$ be a Delzant polytope and let $f_i:\De_i \to \R$ be a continuous function whose epigraph $\widetilde{\De}_i$ is a Delzant polyhedral set. $(\De_1,f_1),(\De_2,f_2)$ have {\bf isomorphic one-skeletons} if there exists a bijection $B: \V(\De_1,f_1) \to \V(\De_2,f_2)$ such that 
\begin{enumerate}
    \item the induced map $B: \E(\De_1,f_1) \to \E(\De_2,f_2)$ preserves directed edges,
    \item $\pi(B(v)) = \pi(v)$ for each $v \in \V(\De_1,f_1) \subset \ft^* \times \R$, and
    \item $\pi(W_{B(v)}) = \pi(W_v)$ for each $v \in \V(\De_1,f_1)$.
\end{enumerate}
\end{definition}
In fact, item (1) is implied by items (2) and (3).

\begin{lemma}\label{lem:preserve-edges}
    For $i=1,2$, let $\De_i\subset \ft^*$ be a Delzant polytope and let $f_i:\De_i \to \R$ be a continuous function whose epigraph $\widetilde{\De}_i$ is a Delzant polyhedral set. Let $B: \V(\De_1,f_1) \to \V(\De_2,f_2)$ be a bijection such that $\pi(B(v)) = \pi(v)$ and $\pi(W_{B(v)}) = \pi(W_v)$ for each $v \in \V(\De_1,f_1)$. Then the induced map $B: \E(\De_1,f_1) \to \E(\De_2,f_2)$ preserves directed edges.
\end{lemma}

\begin{proof}
    Let $(v,w) \in \E(\De_1,f_1)$ be an edge. There exists a weight $\eta \in W_v$ such that $w-v = k\eta$ for some $k>0$. Since the line segment $[v,w]$ is an edge of $\Tilde{\De}_1$ and $\pi:\ft^* \times \R \to \ft^*$ is injective when restricted to $\{(\alpha,f_1(\alpha)): \alpha \in \De_1\}$, $\min\{t>0: \pi(v)+t\pi(\eta) \in\pi (\V(\Tilde{\De}_1)) \}=k$. Since $\pi(W_v) = \pi (W_{B(v)})$, there exists a weight $\eta' \in W_{B(v)}$ such that $\pi(\eta) = \pi(\eta')$. Since $\eta'\in W_{B(v)}$, by definition, there exists a positive number $l$ such that $w':=B(v) + l\eta' \in \V(\Tilde{\De}_2)$ is a vertex of $\Tilde{\De}_2$. Since the line segment $[B(v),w']$ is an edge of $\Tilde{\De}_2$ and $\pi:\ft^* \times \R \to \ft^*$ is injective when restricted to $\{(\alpha,f_2(\alpha)): \alpha \in \De_2\}$, $\min\{t>0: \pi(B(v))+t\pi(\eta') \in\pi (\V(\Tilde{\De}_2)) \}=l$. 
     Since $\pi(B(w)) = \pi(w), \pi(B(v)) = \pi(v)$, we have $\pi(B(w)) =\pi(B(v))+k\pi(\eta')\in\pi (\V(\Tilde{\De}_2))$, so $l \leq k$. If $l < k$, then $\pi(w') \in \pi([v,w])$, so $w' \notin \V(\De_2)\times \R$ and it follows that $w' \in \V(\De_2,f_2)$. Since $B$ is a bijection, there exists a vertex $v' \in \V(\De_1,f_1)$ such that $B(v')= w'$. Then $\pi(B(v')) = \pi(w') =  \pi(v)+l\pi(\eta) \in \pi(\V(\Tilde{\De}_1))$, so $k \leq l$, leading to a contradiction. We conclude that $k=l$. Now, $\pi(B(w)) =\pi(w')$. Since $\V(\Tilde{\De}_2) \cap \{\alpha\} \times \R$ has at most one element for each $\alpha \in \ft^*$, we conclude that $B(w) = w'$. Hence, $(B(v),B(w)) \in \E(\De_2,f_2)$.
\end{proof}

\begin{lemma}\label{lem:tall-space}
    Let $\De \subset \ft^*$ be a Delzant polytope and let $f: \De \to \R$ be a continuous function whose epigraph is a Delzant polyhedral set $\Tilde{\De}$.  Let $(M,\omega,\widetilde{\Phi})$ be a symplectic toric $(T \times S^1)$-manifold such that $\Tilde{\Phi}: M \to \widetilde{\De}$ is a $(T \times S^1)$-quotient map. Then $(M,\omega, \pi \circ \Tilde{\Phi})$ is a connected tall complexity one $T$-space. Consequently, for each $v \in \V(\De,f)$, $\pi(W_v)$ is the set of the isotropy weights of a fixed point of a tall complexity one $T$-space.
\end{lemma}

\begin{proof}
    Since $\Tilde{\De}$ is connected and $\Tilde{\Phi}$ induces a homeomorphism from $M/(T \times S^1)$ to $\widetilde{\De}$, $M$ is also connected. For any $p \in M$, $\Tilde{ \De} \cap (\Phi(p) \times \R)$ contains more than one element, so $\Phi^{-1}(\Phi(p))$ contains more than one $T$-orbit and it follows that $p$ is a tall point. By definition, each $v \in \V(\De,f)$ is the moment image of a $(T\times S^1)$-fixed point, so by Theorem~\ref{thm:lnf} $\pi(W_v)$ is the multiset of the $T$-isotropy weights at the fixed point $\Tilde{\Phi}^{-1}(v)$. Since $(M,\omega, \pi \circ \Tilde{\Phi})$ is tall, Lemma~\ref{lem:defining-monomial} implies that no two weights are equal to each other, so $\pi^{-1}(\alpha) \cap W_v$ has at most one element for any $\alpha \in \ft^*$ and thus $\pi(W_v)$ is a set instead of a multiset.
\end{proof}

\begin{lemma}\label{lem:isom-one-skeleton}
     For $i=1,2$, let $\De_i \subset \ft^*$ be a Delzant polytope, and let $f_i: \De_i \to \R$ be a continuous function whose epigraph $\widetilde{\De}_i$ is a Delzant polyhedral set. Let $(M_i,\omega_i,\Tilde{\Phi}_i)$ be a symplectic toric $(T \times S^1)$-manifold such that $\Tilde{\Phi}_i: M_i \to \Tilde{\De}_i$ is a $(T \times S^1)$-quotient map. If the tall complexity one $T$-spaces $(M_1,\omega_1,\pi \circ\Tilde{\Phi}_1)$ and $(M_2,\omega_2,\pi \circ\Tilde{\Phi}_2)$ have isomorphic one-skeletons, then $(\De_1,f_1)$ and $(\De_2,f_2)$ have isomorphic one-skeletons.
\end{lemma}

\begin{proof}
    It is straightforward to verify that a point $p_i \in M_i$ is fixed by $T$ if and only if $\Tilde{\Phi}_i(p_i) \in \V(\De_i,f_i) \sqcup (\V(\De_i) \times \R)$. If $\Tilde{\Phi}_i(p_i) \in \V(\De_i) \times \R$, then $p_i$ is non-exceptional, as every point in the same $T$-moment fiber is also a $T$-fixed point. On the other hand, if $\Tilde{\Phi}_i(p_i) \in \V(\De_i,f_i)$, then $\pi(\Tilde{\Phi}_i(p_i))$ is not a vertex of $\De_i$, so $p_i$ is exceptional (c.f.~\cite[Example 1.4]{KT14}). Hence, $\Tilde{\Phi}_i$ induces a bijection from the set of exceptional fixed points to $\V(\De_i,f_i)$.
    
    Hence, the isomorphism between the one-skeletons of complexity one $T$-spaces induces a bijection $B: \V(\De_1,f_1) \to \V(\De_2,f_2)$. Since the isomorphism intertwine the moment maps, $\pi(B(v)) = \pi(v)$ for all $v \in \V(\De_1,f_1)$. Moreover, by Lemma~\ref{lem:tall-space}, for any $v \in \V(\De_1,f_1)$, the isotropy weights of the $T$-fixed point $\Tilde{\Phi}_i^{-1}(v)$ is given by $\pi(W_v)$. Hence, $\pi(W_v) = \pi(W_{B(v)})$ as the isomorphism also preserves the slice representations. By Lemma~\ref{lem:preserve-edges}, $(\De_1,f_1)$ and $(\De_2,f_2)$ have isomorphic one-skeletons.
\end{proof}

\begin{lemma}\label{lem:isom-Delzant}
    For $i=1,2$, let $\De_i\subset \ft^*$ be Delzant polytopes and let $f_i:\De_i \to \R$ be a continuous function whose epigraph $\widetilde{\De}_i$ is a Delzant polyhedral set. If $(\De_1,f_1),(\De_2,f_2)$ have isomorphic one-skeletons, then there exist $A \in \ell$ and $c \in \R$ such that $\Gamma (\De_1,f_1+\langle A, \cdot \rangle + c) = \Gamma(\De_2,f_2)$.
\end{lemma}

\begin{proof}
    Let $B: \V(\De_1,f_1) \to \V(\De_2,f_2)$ be the bijection that induces an isomorphism between the one-skeletons. Fix $v \in \V(\De_1,f_1)$. Since both $\widetilde{\De}_1,\widetilde{\De}_2$ are Delzant polyhedral sets, $W_v, W_{B(v)}$ are both bases of $\ell^* \oplus \Z$, so there exists a linear isomorphism $L_v \in \GL(\ell^* \oplus \Z)$ mapping $W_v$ to $W_{B(v)}$. Since $\pi(W_v) = \pi(W_{B(v)})$, we can choose $L_v$ such that $\pi(L_v(\eta)) = \pi(\eta)$ for any $\eta \in \ft^* \times \R$. In other words, we can write $L_v$ in the block form as $L_v = \begin{bmatrix}
        I_d & 0 \\ A &  1
    \end{bmatrix}$ for some $A \in \ell$.
    Define $\Tilde{B}(\Tilde{\alpha}) = L_v(\Tilde{\alpha}-v) +B(v)$, where $\Tilde{\alpha} \in \ft^* \times \R$. Notice that since $\pi(B(v)) = \pi(v)=\pi(L_v(v))$, there exists a constant $c\in \R$ such that for $(\alpha,x) \in \ft^* \times \R$, we have $\Tilde{B}(\alpha,x) = (\alpha, \langle A,\alpha\rangle +x+c)$. It follows that $\Tilde{\De}_1' := \Tilde{B}(\Tilde{\De}_1)$ is the epigraph of the continuous function $f_1+\langle A, \cdot \rangle + c$. Hence, $\Gamma (\De_1,f_1+\langle A, \cdot \rangle + c) = (\Tilde{B}(\V(\De_1,f_1)), \Tilde{B}(\E(\De_1,f_1)), L_v(W(\De_1,f_1))$. It suffices to show that $\Tilde{B}|_{\V(\De_1,f_1)}= B$ and for each $w \in \V(\De_1,f_1)$, $L_v(W_w) = W_{B(w)}$.
    
    For any $w \in \V(\De_1,f_1)$ such that $(v,w) \in \E(\De_1,f_1)$, let $\eta$ denote the weight associated to $(v,w)$. We have that $L_v(\eta) \in W_{B(v)}$. We claim that $L_v(\eta)$ directs the edge $(B(v),B(w))$. By assumption, there exists $\eta' \in W_v$ such that $L_v(\eta')$ directs $(B(v),B(w))$. Hence, $B(w) - B(v) = a L_v(\eta')$ for some $a >0$. Since $\pi(v) = \pi(B(v))$ for all $v$, and since $ \pi(L_v(\eta')) = \pi(\eta')$, applying $\pi$ to both sides of the equation $B(w) - B(v) = a L_v(\eta')$, we have $\pi(w-v) = a\pi(\eta')$. On the other hand, since $\eta$ directs $(v,w)$, there exists $b >0$ such that $ w-v = b \eta$, so $\pi(w-v) = b\pi(\eta)$. Hence, $a \pi(\eta') = b \pi(\eta)$. 
    By Lemma~\ref{lem:tall-space}, $\pi(W_v)$ is the set of isotropy weight at a fixed point of a tall complexity one $T$-space, so by Lemma~\ref{lem:defining-monomial}, $a =b$ and $\eta=\eta'$. Moreover, this implies that $\Tilde{B} (w) = L_v(w-v) + B(v) = aL_v(\eta') + B(v)  = B(w)$.

    Now we show that $L_v(W_w) = W_{B(w)}$. For any $\eta_w \in W_w$, there exists a unique $\eta_v \in W_v$ such that $\eta_w = \eta_v + k\eta$. Hence, $L_v(\eta_w) = L_v(\eta_v) + k L_v(\eta)$. Now let $\eta_{B(w)} \in W_{B(w)}$ be the unique weight such that $\pi(\eta_{B(w)}) = \pi(\eta_w)$. There exists a unique $\eta_{B(v)} \in W_{B(v)}$ such that $\eta_{B(w)} = \eta_{B(v)} + lL_v(\eta)$. 
    Since $\pi(\eta_{B(w)}) = \pi(\eta_w)  = \pi(L_v(\eta_w))$, we have 
    \[\pi(\eta_{B(v)}) + l \pi(L_v(\eta)) =  \pi(L_v(\eta_v)) + k \pi(L_v(\eta)).\]
    Using the same argument as in the last paragraph, we have $\eta_{B(v)} =  L_v(\eta_v)$ and $k=l$. Hence, $L_v(\eta_w) = L_v(\eta_v) + k L_v(\eta) = \eta_{B(v)}+lL_v(\eta) = \eta_{B(w)} \in W_{B(w)}$. This implies that $L_v(W_w) \subseteq W_{B(w)}$. Since $L_v$ is an isomorphism and $W_w,W_{B(w)}$ have the same number of elements, $L_v(W_w) = W_{B(w)}$.

    By induction, $L_v (W_w) = W_{B(w)}$ and $\Tilde{B}(w) = B(w)$ for all $w \in \V(\De_1,f_1)$, and it follows that $(\Tilde{B}(\V(\De_1,f_1)), \Tilde{B}(\E(\De_1,f_1)), L_v(W(\De_1,f_1)) = (\V(\De_2,f_2),\E(\De_2,f_2), W(\De_2,f_2))$, so
    $\Gamma (\De_1,f_1+\langle A, \cdot \rangle + c) = (\Tilde{B}(\V(\De_1,f_1)), \Tilde{B}(\E(\De_1,f_1)), L_v(W(\De_1,f_1)) = \Gamma(\De_2,f_2).$
\end{proof}

Next, we define the skeleton of a pair $(\De,f)$ and prove that it is isomorphic to the skeleton of the corresponding tall complexity one $T$-space. Recall that given  a subset $P$ of $\ft^* \times \R$, the bottom of $P$ is defined to be $\del_-P: = \{(\alpha,x)\in P\mid \text{if }(\alpha,x')\in P\text{ then }x\leq x'\}.$

Let $\De \subset \ft^*$ be a Delzant polytope. Let $f: \De \to \R$ be a continuous function whose epigraph $\widetilde{\De}$ is a Delzant polyhedral set. Let $(M,\omega, \Tilde{\Phi})$ be a symplectic toric $(T \times S^1)$-manifold such that $\Tilde{\Phi}: M \to \Tilde{\De}$ is a $(T \times S^1)$-quotient map. By the last claim of Lemma~\ref{lem:equal-orbits-slice-rep}, the $\Tilde{\Phi}$-moment image of the skeleton of the complexity one $T$-space $(M,\omega, \pi \circ\Tilde{\Phi})$ is a subset of $\del_-\Tilde{\De}$. For this reason, our definition of the skeleton of a pair $(\De,f)$ will only involve the bottom of $\Tilde{\De}$. 

Let $\Gamma(\De,f)$ be the one-skeleton of $(\De,f)$. Let $P = \textrm{Conv}(\V(\De,f))$. Notice that $\del_-P = P \cap \del_-\Tilde{\De}$. To each face $F$ of $P$ such that $F \subseteq \del_-P$ and each vertex $v \in \V(\De,f) \cap F$, we associate a set $W_v^F:=\{\eta \in W_v: \eta \in \R(F-v)\}$. By Lemma~\ref{lem:mod-equal}, for any vertices $v,w \in \V(\De,f) \cap F$, the sets $W_v^F,W_w^F$ span the same lattice. Hence, to each face $F \subset \del_-P$, we can associate a well-defined lattice $L_F:= \Z(W_v^F) \subset \ell^* \oplus \Z$.

\begin{definition}\label{def:de-skeleton}
    Let $\De \subset \ft^*$ be a Delzant polytope. Let $f: \De \to \R$ be a continuous function  whose epigraph $\widetilde{\De}$ is a Delzant polyhedral set. The skeleton $\Sigma(\De,f)$ of the pair $(\De,f)$ is the union of the relative interior\footnote{Recall that given a face $F \subset P$, the {\bf relative interior} of $F$, denoted by $\relint(F)$, is the interior of $F$ inside the affine span of $F$ under the subspace topology.} of faces $F$ of $P = \textrm{Conv}(\V(\De,f))$ such that $F \subset \del_-P$ and one of the following holds:
\begin{enumerate}
    \item $\pi(L_F)$ is not a primitive sublattice of $\ell^*$, or
    \item $\dim \De - \dim F \geq 1$ and $\pi(W_v \setminus W_v^F) \cap \pi(L_F) =\emptyset$ for any $v\in F \cap \V(\De,f)$.
\end{enumerate}
\end{definition}

Before introducing the labels, we first recall that for each $\eta \in \ell^* \oplus \Z$, we can think of $\pi(\eta) \in \ell^*$ as a homomorphism $T \to S^1$, so $\ker(\pi(\eta))$ is a subgroup of $T$. Each $\relint(F)$ in the skeleton is labeled by the following information:
\begin{enumerate}
    \item the stabilizer group $T_F=\cap_{\eta \in W_v^F} \ker(\pi(\eta))$, and
    \item the slice representation: the weights in $\pi(W_v \setminus W^F_v)$ induce a representation of $T$ on $\C^l$, where $l$ is the number of weights in $\pi(W_v \setminus W^F_v)$.
\end{enumerate}

The next lemma implies that every point in $\Sigma(\De,f)$ is the moment image of an exceptional orbit.

\begin{lemma}\label{lem:nonexceptional}
    Fix $p \in \Tilde{\Phi}^{-1}(\del_-P)$ and let $F$ be the unique face such that $\Tilde{\Phi}(p) \in \relint(F)$. $p$ is non-exceptional if and only if $\pi(L_F)$ is a primitive sublattice of $\ell^*$ and one of the following holds
    \begin{enumerate}
        \item $\dim \De = \dim F$, or 
        \item $\pi(W_v \setminus W_v^F) \cap \pi(L_F) \neq \emptyset$ for any $v \in F \cap \V(\De,f)$.
    \end{enumerate}
\end{lemma}
\begin{proof}
    Let $\widetilde{H} \subset T \times S^1$ be the stabilizer group of $p$ under the $(T \times S^1)$-action and let $H:=\widetilde{H} \cap (T \times \{e\})$. By Lemma~\ref{lem:local-cone} and Theorem~\ref{thm:lnf}, $\widetilde{\mathfrak{h}}^\circ =  \R(F-\Tilde{\Phi}(p))$, here $\R(F-\Tilde{\Phi}(p))$ is the linear subspace of $\ft^* \times \R$ parallel to $F$. It is straightforward to check that $H \cong \cap_{\eta \in W_v^F} \ker(\pi(\eta))$ for any $v \in F \cap \V(\De,f)$.

    First, assume that $\pi(L_F)$ is a primitive sublattice of $\ell^*$ and (1) holds. Since $\pi$ is injective when restricted to $\del_-P$, $\dim \pi(L_F) = \dim F = \dim \De$, so $\pi(L_F) = \ell^*$, so $\{\pi(\eta): \eta \in W_v^F\}$ form a basis of $\ell^*$. It follows that $H$ is the trivial group, so $p$ is non-exceptional.

    Now assume that $\pi(L_F)$ is a primitive sublattice of $\ell^*$ and (2) holds, then there exists an isotropy weight $\eta' \in W_v \setminus W_v^F$ such that $\pi(\eta') = \sum_{\eta_i \in W_v^F} a_i \pi(\eta_i)$ for some $a_i \in \Z$. Let $F'$ be the smallest face of $\Tilde{\De}$ such that $F \subsetneq F'$ and $\eta' \in \R(F'-\Tilde{\Phi}(p))$. Since $W_v$ is a $\Z$-basis of $\ell^* \oplus \Z$, $(0,1) = \sum_{\eta_i \in W_v} b_i \eta_i$ for some $b_i \in \Z$, so $ \sum_{\eta_i \in W_v} b_i \pi(\eta_i) = 0$. By Lemma~\ref{lem:tall-space}, $\pi(W_v)$ is the isotropy weights of a fixed point of a tall complexity one $T$-space. Since we also have $\pi(\eta') -\sum_{\eta_i \in W_v^F} a_i \pi(\eta_i)=0$, by Lemma~\ref{lem:defining-monomial}, $b_i=0$ for all $\eta_i \notin W_v^{F'}$. This implies that $(0,1) \in \Z(W_v^{F'})$. For $\epsilon>0$ sufficiently small, $\Tilde{\Phi}(p) +(0,\epsilon) \in \relint(F')$. By the argument in the first paragraph, any point $q$ such that $\Tilde{\Phi}(q) = \Tilde{\Phi}(p) +(0,\epsilon)$ has $T$-stabilizer group $H' \cong \cap_{\eta \in W_v^{F'}} \ker(\pi(\eta))= \cap_{\eta \in W_v^F} \ker(\pi(\eta)) \cap \ker(\pi(\eta')) =  \cap_{\eta \in W_v^F} \ker(\pi(\eta)) \cong H$, where the second equality follows from the fact that $\pi(\eta) \in \pi(L_F) = \Z(\pi(W_v^F))$, so $p$ is non-exceptional.

    Conversely, assume that $p$ is non-exceptional. If $\pi(\Tilde{\Phi}(p)) \in \relint(\De)$, then by~\cite[Example 1.4]{KT14} $H$ is the trivial group. This implies that $\pi(L_F) = \ell^*$ and it follows that $\dim F = \dim (\pi(L_F)) = \dim \ft^* = \dim \De$. If $\pi(\Tilde{\Phi}(p)) \in \partial \De$, then there exists a face $F'$ of $\Tilde{\De}$ such that $F \subsetneq F'$ and $(0,1) \in \R(F'-\Tilde{\Phi}(p))$. Since $p$ is non-exceptional, for sufficiently small $\epsilon>0$, any point $q$ such that $\Tilde{\Phi}(q) = \Tilde{\Phi}(p)+(0,\epsilon)$ has the same $T$-stabilizer as $p$. It follows that $\cap_{\eta \in W_v^F} \ker(\pi(\eta)) = \cap_{\eta \in W_v^{F'}} \ker(\pi(\eta))$. Therefore, $\pi(W_v^{F'} \setminus W_v^F) \subset \pi(L_F)$, so (2) holds. Since $W_v$ is a $\Z$-basis of $\ell^*\oplus \Z$ and $\pi(W_v^{F'}) \subset \Z(\pi(W_v^F))$, there exists a unique $\eta' \in W_v^{F'} \setminus W_v^F$. Since $(0,1) \in \R(F'-\Tilde{\Phi}(p))$, $(0,1) = a\eta'+\sum_{\eta_i \in W_v^F} a_i \eta_i$ for some $a,a_i \in \Z$. This implies that $a\pi(\eta')+\sum_{\eta_i \in W_v^F} a_i \pi(\eta_i)=0$. Since $\pi(\eta') \in \pi(L_F)$, $a = \pm 1$, so after possibly replacing $a_i$ by $-a_i$, we can write $\pi(\eta') = \sum_{\eta_i \in W_v^F} a_i \pi(\eta_i)$. Next, we prove that $\pi(L_F)$ is primitive. Fix $\alpha \in \pi(L_F)$ and a nonzero integer $m$ such that $\alpha =m\beta$ for some $\beta \in \ell^*$. Then, there exist $b_i,c_i \in \Z$ such that $(\beta,0) = \sum_{\eta_i \in W_v} b_i \eta_i$ and $\alpha = \sum_{\eta_i \in W_v^F} c_i\pi(\eta_i)$. So, $m\sum_{\eta_i \in W_v} b_i \pi(\eta_i) = \sum_{\eta_i \in W_v^F} c_i\pi(\eta_i)$. By Lemma~\ref{lem:tall-space} and Lemma~\ref{lem:defining-monomial}, since $\pi(\eta') = \sum_{\eta_i \in W_v^F} a_i \pi(\eta_i)$, the previous equation implies that $b_i=0$ for all $\eta_i \notin W_v^{F'}$ and $c_i = m(b_i+b'a_i)$ for all $\eta_i \notin W_v^{F}$, so $\beta = \frac{\alpha}{m} = \sum_{\eta_i \in W_v^F} (b_i+b'a_i) \pi(\eta_i) \in \pi(L_F)$. Thus, $\pi(L_F)$ is a primitive sublattice of $\ell^*$.
\end{proof}

\begin{proposition}\label{prop:skeleton}
     Let $\De \subset \ft^*$ be a Delzant polytope. Let $f: \De \to \R$ be a continuous function whose epigraph is a Delzant polyhedral set $\widetilde{\De}:= \{(\alpha,x) \in \De \times \R: x \geq f(\alpha)\}$. Let $(M,\omega, \Tilde{\Phi})$ be a symplectic toric $(T \times S^1)$-manifold such that $\Tilde{\Phi}: M \to \Tilde{\De}$ is a $(T \times S^1)$-quotient map. Then the skeleton $\Sigma$ of the complexity one $T$-space  $(M,\omega, \pi \circ\Tilde{\Phi})$ is isomorphic to $\Sigma(\De,f+\langle A,\cdot\rangle+c)$ for any $A \in \ell, c \in \R$.
\end{proposition}

\begin{proof}
    We first prove that $\Tilde{\Phi}$ induces an isomorphism from the skeleton of $(M,\omega, \pi \circ\Tilde{\Phi})$ to $\Sigma(\De,f)$. Let $p$ be an exceptional point in $M$. Let $\widetilde{H} \subset T \times S^1$ be the stabilizer group of $p$ under the $(T \times S^1)$-action and let $H:=\widetilde{H} \cap (T \times \{e\})$. Let $F$ be the minimal face of $\widetilde{\De}$ such that $\Tilde{\Phi}(p) \in \relint(F)$. By Lemma~\ref{lem:local-cone} and Theorem~\ref{thm:lnf}, $\widetilde{\mathfrak{h}}^\circ =  \R(F-\Tilde{\Phi}(p))$, here $\R(F-\Tilde{\Phi}(p))$ is the linear subspace of $\ft^* \times \R$ parallel to $F$.  
    
    By Lemma~\ref{lem:equal-orbits-slice-rep}, $\Tilde{\Phi}(p) \in \del_-\Tilde{\De}$ and the $T$-orbit through $p$ is the same as the $(T \times S^1)$-orbit through $p$, i.e. $(T \times S^1)/\widetilde{H} \cong T/H$. Therefore, $N:=\Tilde{\Phi}^{-1}(F)$ is a symplectic toric $(T/H)$-manifold and $\frac{1}{2}\dim N = \dim F = \dim (T/H)$ (c.f.~\cite[Section 2]{De88}). By construction, $N$ is the connected component of $M^H$ that contains $p$.  Hence, by Lemma~\ref{lem:toric-exceptional}, every point in $N$ is exceptional. In particular, the fixed points in $N$ are all exceptional and $F = \textrm{Conv}(\Tilde{\Phi}(N^T))$. By the proof of Lemma~\ref{lem:isom-one-skeleton}, $\Tilde{\Phi}(N^T) \subseteq \V(\De,f)$. It follows that $\Tilde{\Phi}(p) \in P:=\textrm{Conv}(\V(\De,f))$ and thus $\Tilde{\Phi}(p) \in \del_-P$. Hence, $\Tilde{\Phi}(\Sigma) \subset \del_-P$. By Lemma~\ref{lem:nonexceptional}, $\Tilde{\Phi}(\Sigma) = \Sigma(\De,f)$. Moreover, since for each exceptional point, its $T$-orbit is the same as the $(T\times S^1)$-orbit,  the natural map $M/T \to M/(T \times S^1)$ identifies $\Sigma$ in two quotient spaces. Since $\Tilde{\Phi}$ induces a homeomorphism from $M/(T \times S^1)$ to $\widetilde{\De}$, $\Tilde{\Phi}$ induces a homeomorphism from $\Sigma$ to $\Sigma(\De,f)$. It is straightforward to verify that the stabilizer group label and the slice representation label give the correct $T$-stabilizer group and the slice representation for the corresponding orbit, so $\Tilde{\Phi}$ induces an isomorphism from $\Sigma$ to $\Sigma(\De,f)$.

    Fix $A \in \ell$ and $c \in \R$. Define $\widetilde{\De}(A,c):= \{(\alpha,x) \in \De \times \R: x \geq f(\alpha) + \langle A,\alpha \rangle +c \}$ and let $(M_{(A,c)},\omega_{(A,c)},\Tilde{\Phi}_{(A,c)})$ be a symplectic toric $(T \times S^1)$-manifold  such that $\Tilde{\Phi}_{(A,c)}$ induces a homeomorphism from $M_{(A,c)}/(T\times S^1)$ to $\Tilde{\De}(A,c)$.

    Let $\rho: T \times S^1 \to \textrm{Sympl}(M,\omega)$ be the toric $(T \times S^1)$-action on $(M,\omega)$. Consider the linear isomorphism $L_A:=  \begin{bmatrix}
        I_d & 0 \\ A &  1
    \end{bmatrix}$ on $\ft^* \times \R$. Let $\Lambda \in \GL( \ft \times \R)$ be the dual linear isomorphism and identify $\Lambda$ with the automorphism it induces on $T \times S^1$. Define $\Tilde{\Psi}: M \to \ft^* \times \R$ by $\Tilde{\Psi}(p) = L_A(\Tilde{\Phi}(p)) + (0,c)$. It is straightforward to check that $\Tilde{\Psi}$ is a moment map for the Hamiltonian $(T \times S^1)$-action given by $\rho \circ \Lambda$ and that $\Tilde{\Psi}(M) = \widetilde{\De}(A,c)$. Hence, by Theorem~\ref{thm:KL}, since $\Tilde{\De}(A,c)$ is convex, $(M_{(A,c)},\omega_{(A,c)},\Tilde{\Phi}_{(A,c)})$ is isomorphic to $(M,\omega, \Tilde{\Psi})$. Since $\Lambda|_{T \times \{e\}}$ is the identity map, the identity map on $M$ is a $T$-equivariant symplectomorphism that intertwines $\pi \circ \Tilde{\Psi}$ and $\pi \circ \Tilde{\Phi}$. Hence, there is a $T$-equivariant symplectomorphism between $(M_{(A,c)},\omega_{(A,c)},\pi \circ\Tilde{\Phi}_{(A,c)})$ and $(M,\omega, \pi\circ\Tilde{\Phi})$. The claim follows.
\end{proof}

Now, we are ready to prove Theorem~\ref{thm:1-skeleton}.

\begin{proof}[Proof of Theorem~\ref{thm:1-skeleton}]
     It suffices to prove the theorem by assuming that both spaces have $1$-colorable skeletons. For $i=1,2$, let $\Sigma_i \subset M_i/T$ be the skeleton and let $\De_i: =\Phi_i(M_i)$ be the moment polytope. By Theorem~\ref{thm:inj-toric-alt}, there exists a continuous function $f_i: \De_i \to \R$ such that 
\begin{enumerate}
\item $\Tilde{\De}_i:= \{(\alpha,x) \in \De_i \times \R: x \geq f_i(\alpha)\}$ is a Delzant polyhedral set.
\item There exists a symplectic toric $(T \times S^1)$-manifold  $(\Tilde{M}_i, \Tilde{\omega}_i, \Tilde{\Phi}_i)$ such that $\Tilde{\Phi}_i: \Tilde{M}_i \to \Tilde{\De}_i$ is a $(T \times S^1)$-quotient map, and the skeleton of the tall complexity one $T$-space $(\Tilde{M}_i, \Tilde{\omega}_i, \pi \circ \Tilde{\Phi}_i)$ is isomorphic to $\Sigma_i$.
\end{enumerate}

By assumption, $(M_1,\omega_1,\Phi_1)$ and $(M_2,\omega_2,\Phi_2)$ have isomorphic one-skeletons, so $(\Tilde{M}_1, \Tilde{\omega}_1, \pi \circ \Tilde{\Phi}_1)$ and $(\Tilde{M}_2, \Tilde{\omega}_2, \pi \circ \Tilde{\Phi}_2)$ have isomorphic one-skeletons. By Lemma~\ref{lem:isom-one-skeleton}, $(\De_1,f_1)$ and $(\De_2,f_2)$ have isomorphic one-skeletons. By Lemma~\ref{lem:isom-Delzant}, there exist $A \in \ell$ and $c \in \R$ such that $\Gamma (\De_1,f_1+\langle A, \cdot \rangle + c) = \Gamma(\De_2,f_2)$. By definition, the skeleton of $(\De_1,f_1+\langle A, \cdot \rangle + c)$ is equal to the skeleton of $(\De_2,f_2)$. By Proposition~\ref{prop:skeleton}, the skeleton of the complexity one $T$-space $(\Tilde{M}_1, \Tilde{\omega}_1, \pi \circ \Tilde{\Phi}_1)$ is isomorphic to the skeleton of $(\De_1,f_1+\langle A, \cdot \rangle + c)$ and the skeleton of $(\Tilde{M}_2, \Tilde{\omega}_2, \pi \circ \Tilde{\Phi}_2)$ is isomorphic to the skeleton of $(\De_2,f_2)$. Hence, $(M_1,\omega_1,\Phi_1)$ and $(M_2,\omega_2,\Phi_2)$ have isomorphic skeletons.
\end{proof}

\bibliographystyle{alpha}
\bibliography{ref}

\end{document}